\documentclass[10pt]{article}
\usepackage{}
\usepackage{amsmath}
\allowdisplaybreaks[4]
\usepackage{amsthm}
\usepackage{amsfonts}
\usepackage{fancybox}
\usepackage{epsfig}
\usepackage{epsf}
\usepackage{amssymb}
\usepackage{epic,eepic}
\usepackage{latexsym,bm}
\usepackage{graphicx}
\usepackage{multirow}
\usepackage{subfigure}
\usepackage{float}
\usepackage{color}
\usepackage[top=0.8in, bottom=0.8in, left=0.8in, right=0.8in, dvips, letterpaper]{geometry}
\usepackage{algorithm}
\usepackage{algorithmic}
\usepackage{cite}
\usepackage{boxedminipage}

\usepackage{url}
\usepackage{lineno} 
\usepackage{makecell}

\newcommand\argmin{\mathop{\textrm{argmin}}}
\newcommand\crule[1][5cm]{%
  \par
  \nointerlineskip
  \centerline{\hbox to #1{\hrulefill}}%
  \nointerlineskip}

\numberwithin{equation}{section}
\numberwithin{algorithm}{section}
\newtheorem{theorem}{{\sc Theorem}}[section]
\newtheorem{lemma}{{\sc Lemma}}[section]
\newtheorem{corollary}[theorem]{Corollary}
\newtheorem{remark}{Remark}[section]
\newtheorem{assumption}{Assumption}[section]
\newtheorem{proposition}{Proposition}[section]
\newtheorem{definition}{{\sc Definition}}[section]
\newtheorem{example}{{Example}}[section]

\newcommand{\R}{\mathbb{R}}

\newcommand{\be}{\begin{equation}}
\newcommand{\ee}{\end{equation}}
\newcommand{\bee}{\begin{equation*}}
\newcommand{\eee}{\end{equation*}}
\newcommand{\bea}{\begin{eqnarray}}
\newcommand{\eea}{\end{eqnarray}}
\newcommand{\beaa}{\begin{eqnarray*}}
\newcommand{\eeaa}{\end{eqnarray*}}

\title{
	\textbf{Nonconvex Nonsmooth Multicomposite Optimization and Its Applications to Recurrent Neural Networks}}

\author{Lingzi Jin\thanks{{ling-zi.jin@connect.polyu.hk}, Department of Applied Mathematics, Hong Kong Polytechnic University, Kowloon, Hong Kong.  }
        \and
        Xiao Wang\thanks{{wangx936@mail.sysu.edu.cn}, School of Computer Science and Engineering, Sun Yat-sen University, Guangzhou, China. This author is partially supported by National Natural Science Foundation of China (No 12271278).}
       	\and
       	Xiaojun Chen\thanks{{maxjchen@polyu.edu.hk}, Department of Applied Mathematics, Hong Kong Polytechnic University, Kowloon, Hong Kong. This author is partially supported by Hong Kong Research Grant Council PolyU15300123, JLFS/P-501/24, CAS Joint Laboratory of Applied Mathematics.}
        }

\date{}

\begin{document}
	

\maketitle

\begin{abstract}
We consider a class of nonconvex nonsmooth multicomposite optimization problems where the objective function consists of a Tikhonov regularizer and a composition of multiple nonconvex nonsmooth component functions.
Such optimization problems arise from tangible applications in machine learning and beyond.
To define and compute its first-order and second-order d(irectional)-stationary points effectively, we first derive the closed-form expression of the tangent cone for the feasible region of its constrained reformulation.
Building on this, we establish its equivalence with the corresponding constrained and $\ell_1$-penalty reformulations in terms of global optimality and d-stationarity.
The equivalence offers indirect methods to attain the first-order and second-order d-stationary points of the original problem in certain cases.
We apply our results to the training process of recurrent neural networks (RNNs).
\vspace{0.8cm}


\vspace{0.5cm}


\end{abstract}

\section{Introduction}\label{sec1}
In this paper, we consider the following unconstrained nonconvex nonsmooth optimization problem
\begin{align}
	\label{eq:1.1}
	\tag{P}
	\min_{\theta \in \R^{n} }~ \Psi(\theta) + \lambda \|\theta\|^2 ,
\end{align}
where $\lambda>0$, $\|\cdot\|$ is the Euclidean norm, and the mapping $ \Psi: \R^{n} \rightarrow \R_{+} $ is defined by
\begin{align*}
	& \Psi(\theta) := g(u_{1}, \cdots, u_{L})  \\
	& \text{with }  u_{1}:=\psi_{0}(\theta) \text{ and }
	u_\ell := \psi_{\ell-1}(\theta, u_1, \cdots, u_{\ell-1}), \;\;\,\ell = 2,\cdots, L,
\end{align*}
for $L+1$ continuous but possibly nonconvex nonsmooth component functions
$$\psi_{\ell-1}: \R^{n+ \bar{N}_{\ell-1} } \rightarrow \R^{N_{\ell}}, \quad  \ell =1,\dots,L, \quad {\rm and} \quad g : \R^{\bar{N}_{L}} \rightarrow \R_{+}$$
with $ \bar{N}_{0}:=0 $ and $ \bar{N}_{\ell}:= \sum_{j=1}^{\ell} N_{j} $ for all $ \ell = 1,\dots,L $.
Problem \eqref{eq:1.1} covers a wide range of applications in machine learning where $\theta$ refers to the network parameter, $\Psi$ is the loss function and $\lambda \|\theta\|^2$ is the regularizer to guarantee the boundedness of the solution set \cite{LRP-LLC} and alleviate the overfitting \cite{WZC2024} for \eqref{eq:1.1}.

In \cite{DNN-CHP}, Cui et al. present a novel deterministic algorithmic framework that enables the computation of a d-stationary point of the empirical deep neural network training problem formulated as a multicomposite optimization problem.
The model \eqref{eq:1.1} differs from the model (2.1)-(2.2) of \cite{DNN-CHP} in two aspects.
The first difference is that we unify parameters $ \{\theta_{0},\dots,\theta_{L-1}\} $ (corresponding to $ \{ z_{1},\dots,z_{L} \} $ in \cite{DNN-CHP}) as $\theta$ since the process of selecting $ \theta_{\ell-1} $ from $ \theta $
can be achieved by $\psi_{\ell-1}$, which facilitates the sharing of parameters across layers simultaneously.
Secondly, we explicitly articulate the information transmission across multiple layers (i.e., dependence of $\psi_{\ell-1}$ on $ u_{1},\dots,u_{\ell-2} $), which is widely used in network structure, such as RNN  \cite{Graves2012SupervisedSL} and shortcut in Resnet. In \cite{DNN-CHP} it assumes that $ g $ only depends on $ u_{L} $ and $\psi_{\ell-1}$ only depends on $ (\theta,  u_{\ell-1}) $.
Although $ \Psi $ in \eqref{eq:1.1} can be reorganized into $
\Psi(\theta)= \bar{g}  ( \cdot ) \circ \bar{\psi}_{L-1} ( \theta,  \cdot ) \circ \cdots \circ \bar{\psi}_{0} ( \theta) $ with some functions $ \bar{g}, \{ \bar{\psi}_{\ell} \} $ coinciding the formulation (2.1)-(2.2) in \cite{DNN-CHP} if neglecting the first difference, it can be found that the required number of auxiliary variables under this decomposition is much larger than $ \bar{N}_{L} $.
We illustrate the differences by an example of RNNs in Remark \ref{remark-[8]} with Figure 1.
Thus, model \eqref{eq:1.1} encompasses the formulation (2.1)-(2.2) presented in \cite{DNN-CHP}.

Directly solving \eqref{eq:1.1} by SGD-type methods (SGDs) is common in computer science. However, the automatic differentiation (AD), the key of SGDs, based on chain rules fails for the subdifferential of $ \Psi $ at a nondifferentiable point $\theta$ \cite{BP2021}.
To the best of our knowledge, existing algorithms that directly solve unconstrained  nonconvex nonsmooth problem \eqref{eq:1.1} with rigorous convergence analysis can be roughly separated into two groups. One combines (S)GDs with smoothing techniques aiming at (approximate) Clarke stationary points  \cite{chen2012smoothing,LZJ2022,ZLJSJ2020}.
Another approach constructs advanced AD algorithms based on chain rules for some generalized subdifferentials.
The latter further branches into two distinct paths.
Along the first path Nesterov \cite{Nesterov2005} utilizes the chain rule of directional derivatives to define lexicographic differentiation and evaluate lexicographic subdifferential \cite{BKSW2018,KB2015}.
However, the nice properties of lexicographic subdifferential \cite{KB2015,Nesterov2005} seem to be mostly applied in sensitivity analysis and have not helped to develop an algorithm converging to a stationary point defined by lexicographic subdifferential or a d-stationary point. Moreover, it is mentioned in \cite{BLPS2021} that the AD method based on lexicographic differentiation is incompatible with existing AD frameworks.
Therefore, Bolte and Pauwels \cite{BP2021} follow a path of conservative field, which is a generalization of Clarke subdifferential.
Further study establishes convergence of SGDs in the sense of conservative field stationarity, which can be improved to Clarke stationarity under certain conditions. More references can be referred to \cite{XHLT2024}.
From the existing literature, directly solving \eqref{eq:1.1} may only be able to find a  Clarke stationary point if there is no special structure, such as
weak convexity \cite{LMZ2023} and weak concavity \cite{ALT2025}.

However, in general, Clarke stationarity may be an overly lenient condition in contrast with d-stationarity \cite{nMM-CPS}.
On the other hand, the d-stationary points of multicomposite optimization \eqref{eq:1.1} are too complicated to calculate directly (Proposition \ref{dd_F1}). Therefore, a more practical approach is to reformulate \eqref{eq:1.1} to derive a model with easily computable d-stationary points, while establishing their relationship in terms of d-stationarity.
In \cite[Section 9.4.2]{CP-book}, the equivalence between \eqref{eq:1.1} with $ L=1 $ and its $\ell_{1}$-penalty form in d-stationarity is established under the premise of feasibility.
In \cite{DNN-CHP}, a one-sided relation is obtained for simplified \eqref{eq:1.1} with $ L>1 $ and its $\ell_{1}$-penalty form, which provides the algorithm for calculating d-stationary points of DNN training problem.
More references that establish and utilize the relationship between the simplified \eqref{eq:1.1} and its different reformulations in other kinds of stationarity can be referred to \cite{LReLU-LLC,WB2023,WZC2024}.

Apart from the above first-order optimality conditions, the second-order optimality conditions for nonsmooth optimization problems have attracted widespread interest since the 1970s \cite[Chapter 13]{rockafellar2009variational}.
To avoid the concept of second-order tangent cone, Cui et al. \cite{cui2018studyv2} use a kind of second-order subderivative \cite[13(7)]{rockafellar2009variational} to establish second-order conditions for minimizing twice semidifferentiable and locally Lipschitz continuous functions with polyhedral constraints \cite[Proposition 2.3]{cui2018studyv2}, and apply the results on piecewise linear-quadratic programs.
Jiang and Chen \cite[Lemma 3.8]{min-maxR2} further extend the second-order necessary condition to convexly constrained optimization problems with twice semidifferentiable objective functions, and apply the results on minimax problems by using generalized directional derivatives and subderivatives.
For \eqref{eq:1.1} with $ L=1 $ and twice semidifferentiable component functions, \cite[Proposition 9.4.2]{CP-book} offers second-order conditions by the relation between the original problem and its $ \ell_{1} $-penalty reformulation, and the structure of the reformulation. However, the aforementioned second-order conditions are inapplicable to \eqref{eq:1.1} with $L>1$ even when $ g $ and $ \{\psi_{\ell} \} $ are all twice semidifferentiable, since the composition of such functions may not retain this property.
More references that establish second-order optimality conditions by other generalized Hessians and generalized second-order directional derivatives can be referred to commentary at the end of \cite[Chapter 13]{rockafellar2009variational}.

\subsection{Model reformulation}
Motivated by \cite{CpW2014,DNN-CHP}, we reformulate \eqref{eq:1.1} as a constrained optimization problem.
First we introduce auxiliary variables
\begin{equation}
	\label{def:1.1}
	\begin{aligned}
		&  \mathbf{u}_{\ell}:= (u_{1}^{\top}, \dots, u_{\ell}^{\top})^{\top} \in \R^{\bar{N}_{\ell}}, \, \ell = 1,\dots,L,
		\text{ and an empty placeholder } \mathbf{u}_{0}  
	\end{aligned}
\end{equation}
to decompose the nested structure of $ \Psi $, obtaining the constrained form
\begin{equation}
	\label{eq}
	\tag{P0}
	\begin{aligned}
		\min_{ z } ~& F(z):= g(u) +  \lambda \|\theta \|^{2}  , ~~
		\text{subject to }~
		u_{\ell}= \psi_{\ell-1} ( \theta, \mathbf{u}_{\ell-1} ) , \ell = 1, \dots , L,
	\end{aligned}
\end{equation}
where for brevity we denote $u := \mathbf{u}_{L} \in \R^{\bar{N}_{L}} $,
\begin{align}
	z:=(\theta^{\top}  ,u^{\top} )^{\top}\in \R^{\bar{N}} ,\, \text{and }
	\bar{N}:= n + \bar{N}_{L}.  \label{def:z}
\end{align}

The nonconvex nonsmooth objective function and the nonsmooth equality constraints in \eqref{eq} pose significant challenges for both theoretical analysis and numerical tractability.
Therefore, \eqref{eq} will only be used as an intermediary.
Denote $ [L]:=\{1,\dots,L\} $.
As the final reformulation, the $\ell_{1} $-penalty form of \eqref{eq} with positive penalty parameters $ \{ \beta_{\ell}, \ell \in [L] \} $ is defined as:
\begin{align}
	\label{l1pen}
	\tag{P1}
	\min_{z}~ \Theta(z):= F(z) + \sum_{\ell=1}^{L} \beta_{\ell} \| u_{\ell} - \psi_{\ell-1} ( \theta, \mathbf{u}_{\ell-1} )  \|_{1}.
\end{align}
We will analyze the properties of \eqref{eq:1.1}, \eqref{eq} and \eqref{l1pen} and establish the relationship between them, which makes it realistic to attain second-order stationary points of \eqref{eq:1.1}.

\subsection{Contribution}\label{sec1.1}
The contributions of this paper lie in threefold.

Firstly, we obtain a full characterization of the tangent cone of the feasible region of \eqref{eq} under directional differentiability and local Lipschitz continuity of $g$ and $ \{ \psi_{\ell-1},\ell\in[L] \} $ in Theorem \ref{th:T-express-2}.
In general, it is challenging to express the tangent cone of a nonconvex feasible region \cite[p525 and Remark 9.2.1]{CP-book}.
For the nonconvex feasible region constructed by nonsmooth equality constraints in \eqref{eq}, it can be verified that NNAMCQ (no nonzero abnormal multiplier constraint qualification) \cite[Remark 2]{Y2000NNAMCQ} holds using the method similar to \cite[Lemma 6]{LReLU-LLC}.
Based on that, a subset of its tangent cone can be expressed by a superset of its normal cone \cite[Corollary 10.50]{rockafellar2009variational} using the relations between tangent and normal cones \cite[Theorems 6.26 and 6.28]{rockafellar2009variational}.
However, the closed-form of its tangent cone is still difficult to obtain solely through constraint qualifications (CQs).
In contrast, we provide a closed-form expression of the tangent cone of the feasible region of \eqref{eq} by directly utilizing the pull-out structure of constraints.

Secondly, we show the equivalence between  \eqref{eq:1.1}, \eqref{eq} and \eqref{l1pen} regarding d-stationary points and global minimizers, which generalizes the results in \cite{DNN-CHP} and Chapter 9 of \cite{CP-book}.
As a consequence of the equivalence between \eqref{eq:1.1} and \eqref{l1pen}, the penalty form \eqref{l1pen} with according algorithms \cite{nMM-CPS,Y141} offers an alternative way to solve the original problem \eqref{eq:1.1}.
Furthermore, we derive a unified second-order necessary condition for nonconvex nonsmooth constrained minimization with twice directionally differentiable objective functions, which extends the results in \cite{cui2018studyv2,min-maxR2}.
Together with the equivalence between \eqref{eq:1.1}, \eqref{eq} and \eqref{l1pen}, the second-order optimality conditions for \eqref{eq} and \eqref{l1pen} provide second-order necessary and sufficient criteria for \eqref{eq:1.1}, which cover the ones proposed in \cite[Proposition 9.4.2]{CP-book}. 

Thirdly, we apply our theoretical results to the minimization problem for training an Elman RNN with a single unidirectional hidden layer.
The equivalence in d-stationarity of \eqref{form2-eq} and \eqref{form2-l1pen} not only generalizes the result from Theorem 2.1 of \cite{DNN-CHP}, but also provides the explicit thresholds for penalty parameters.
Moreover, we observe that every d-stationary point of \eqref{form2-eq} is also a second-order d-stationary point for \eqref{form2-eq} and the same result holds for \eqref{form2-l1pen} under certain conditions, which makes their second-order d-stationary points computable by the methods for DC programs \cite{nMM-CPS}.

\subsection{Organization}\label{sec1.2}
The rest of this paper is organized as follows.
In Section \ref{sec2}, we introduce some basic definitions and preliminary properties of \eqref{eq:1.1}, \eqref{eq} and \eqref{l1pen}. 
The d-stationarity of \eqref{eq:1.1}, \eqref{eq}, \eqref{l1pen} and the second-order d-stationarity of \eqref{eq}, \eqref{l1pen} are defined in Section \ref{sec3}.
Based on the closed-form expression of the tangent cone of the feasible region of \eqref{eq} in subsection \ref{sec3.1},
we establish the equivalence between \eqref{eq:1.1}, \eqref{eq} and \eqref{l1pen} in terms of global optimality and d-stationarity under certain conditions in subsection \ref{sec3.2}. And subsection \ref{sec3.3} shows that second-order d-stationarity of \eqref{eq} and \eqref{l1pen} provides second-order necessary conditions for \eqref{eq:1.1}.
Subsection \ref{sec3.4} offers second-order sufficient conditions for strong local minimizers of \eqref{eq:1.1} through \eqref{l1pen}. 
In Section \ref{sec4}, we apply the general theoretical results to RNNs.
Concluding remarks are given in Section \ref{sec5}.

\subsection{Notation}\label{sec1.3}
In the following, we denote the set of integers and nonnegative (positive) integers as $ \mathbb{Z} $ and $ \mathbb{Z}_{+} $ ($ \mathbb{Z}_{++} $) respectively.
For any $ m \in \mathbb{Z}_{++} $, we denote $ [m]:=\{ 1,\dots,m  \} $.
The accumulative multiplication is presented by $\prod$.
For any sequence $ \{ a_{j} \geq 0, j\in \mathbb{Z}_{+}  \} $ and any $ j_{1},j_{2} \in \mathbb{Z}_{+}  $ with $ j_{1}>j_{2} $, denote $ \sum_{j=j_{1}}^{j_{2}} a_{j} :=0  $ and $ \prod_{j=j_{1}}^{j_{2}} a_{j} :=1  $.
For any vector sequence $ \{ u_{j}  , j\in \mathbb{Z}_{+}  \} $ and any $ j_{1},j_{2} \in \mathbb{Z}_{+}  $ with $ j_{1}>j_{2} $, denote $ (u_{j_{1}}, \dots, u_{j_{2}}  ) $ as an empty placeholder.
For any vector $a$ and positive integer $i$, $ [a]_{i} $ refers to the $i$th component of $a$.
For any two sets $A,B \subseteq \R^{m} $, denote $ A+B= \{ a+b \mid a\in A , b \in B \} $.
Denote $\mathbb{B}(\mathbf{0};1):= \{ z\in \R^{\bar{N}} \mid \|z\| \leq 1 \} $.
For any set $ \mathcal{F} \subseteq \R^{m} $, the indicator function is defined as
$\delta_{\mathcal{F}}(x)= 0 $, if $ x\in \mathcal{F} $, and $ + \infty $, otherwise.
For any $ m \in \mathbb{Z}_{++}, \gamma \in \R $ and any function $ f:\R^{m} \rightarrow \R\cup \{ + \infty \} $, the level set is defined as $ lev_{\leq \gamma} f:=\{ x \in \R^{m} \mid f(x) \leq \gamma \} $.

Denote the optimal solution sets of \eqref{eq:1.1}, \eqref{eq} and \eqref{l1pen} by  $$\mathcal{S} :=   \argmin_{\theta \in \R^{n} }
~  [ \Psi(\theta) + \lambda\|\theta\|^2 ],
\,\,\,\,\,
\mathcal{S}_{0}:=  \argmin_{z \in  \mathcal{F}_{0}} F(z), \, \,\,\,\,
\mathcal{S}_{1}:=   \argmin_{z \in \R^{\bar{N}} }
~  \Theta( z),$$
respectively,
where
\begin{align}\label{def:set-F0}
	\mathcal{F}_{0} :=  \lbrace z \in \R^{\bar{N}} \mid
	u_{\ell}=\psi_{\ell-1}( \theta,  \mathbf{u}_{\ell-1} ), \ell \in [L]
	\rbrace .
\end{align}

\section{Preliminaries}\label{sec2}
In this section, we present some preliminaries that are used in subsequent sections. 
Let
\begin{equation}\label{def:gammabar}
	\begin{aligned}
		z^{0}&:=( \mathbf{0}^{\top} , (u_{1}^{0})^{\top}, \dots, (u_{L}^{0})^{\top} )^{\top}  \text{ with }  u_{\ell}^{0}:= \psi_{\ell-1}(\mathbf{0}, u_{1}^{0}, \dots, u_{\ell-1}^{0}),\, \ell=1,\dots,L, \\
		\bar{\gamma}&:= \Theta ( z^{0} ).
	\end{aligned}
\end{equation}
Then we have $ z^{0} \in \mathcal{F}_{0}  \neq \emptyset  $ and $ \bar{\gamma}=F(z^{0}) $. 
Next, we prove that $ \mathcal{S}, \mathcal{S}_{0} $ and $ \mathcal{S}_{1} $  are nonempty and compact under the continuity of $ g: \R^{\bar{N}_L} \rightarrow \R_{+} $ and $ \{ \psi_{\ell-1}: \R^{n+\bar{N}_{\ell-1}} \rightarrow \R^{N_{\ell}} ,\ell\in[L] \} $.
Noting the nonnegativity of $g$, we can obtain the following result by the level-boundedness \cite[Theorem 1.9]{rockafellar2009variational} of $  ( \Psi(\cdot) + \lambda\|\cdot\|^2   ) $ and $ (F+\delta_{\mathcal{F}_{0}}) $.
\begin{lemma}\label{lem:bd-opt-0}
	The optimal solution sets 	$ \mathcal{S} $ and  $ \mathcal{S}_{0} $ are nonempty and compact.
\end{lemma}

In fact, it can be naturally obtained that \eqref{eq:1.1} is equivalent to \eqref{eq} in global optimality. If $ \bar{\theta} \in \mathcal{S} $, then $ \bar{z}:= ( \bar{\theta}^{\top}, \bar{u}_{1}^{\top}, \dots , \bar{u}_{L}^{\top} )^{\top} \in \mathcal{S}_{0} $ where $\bar{u}_{\ell}:= \psi_{\ell-1} ( \bar{\theta} ,  \bar{u}_{1}, \dots , \bar{u}_{\ell-1} ) $ for all $ \ell\in[L]$; conversely, if $ \bar{z}:= ( \bar{\theta}^{\top}, \bar{u}_{1}^{\top}, \dots , \bar{u}_{L}^{\top} )^{\top} \in \mathcal{S}_{0} $, then $  \bar{\theta} \in \mathcal{S} $.

\begin{lemma} \label{lem:bd-opt-1}
	The optimal solution set 	$ \mathcal{S}_{1} $ is nonempty and compact.
\end{lemma}
\begin{proof}
	Since $ \Theta $ is proper and continuous, we only need to show its level boundedness \cite[Theorem 1.9]{rockafellar2009variational}. 
	For any $ \gamma \in \R_+ $ and any $ z \in lev_{\leq \gamma}  \Theta  $, it follows from $ g(\cdot)\geq 0 $ and $ \|\cdot\| \leq \|\cdot\|_{1} $ that
	\begin{align}
		& \| \theta  \|  \leq \sqrt{{\gamma}/{\lambda}}, \label{bd1-1} \\
		&   \| u_{\ell} -  \psi_{\ell-1}(\theta, u_{1}, \dots, u_{\ell-1})  \|  \leq  {\gamma}{ \beta_{\ell}^{-1} } , \, \forall \ell\in [L]  .  \label{bd4-1}
	\end{align}
	Next, we will finish the proof in an inductive manner.
	For $\ell=1$, it follows from \eqref{bd1-1}-\eqref{bd4-1} and the continuity of $ \psi_{0} $ on $ \R^{n} $ that
	\begin{align*}
		\|u_{1}\|
		\leq  \| u_{1} -  \psi_{0}(\theta )  \| + \| \psi_{0}(\theta )  \|
		\leq  {\gamma}{ \beta_{1}^{-1} } + \max \{ \psi_{0}(\theta ) \mid \| \theta  \|  \leq \sqrt{ {\gamma}{\lambda^{-1}}} \}
		< + \infty.
	\end{align*}
	For any $ \ell=2,\dots,L $, assume that $ u_{1},\dots,u_{\ell-1} $ are bounded. Then it follows from \eqref{bd1-1} and \eqref{bd4-1} that
	\begin{align*}
		\|u_{\ell}\|
		&\leq  \| u_{\ell} -  \psi_{\ell-1}(\theta ,u_{1},\dots,u_{\ell-1})  \| + \| \psi_{\ell-1}(\theta ,u_{1},\dots,u_{\ell-1})  \| \\
		&\leq
		{\gamma}{ \beta_{\ell}^{-1} } + \| \psi_{\ell-1}(\theta ,u_{1},\dots,u_{\ell-1})  \| 
		< + \infty.
	\end{align*}
	Hence, $u$ is bounded by induction. Together with \eqref{bd1-1} and arbitrariness of $z$, it implies the boundedness of $ lev_{\leq \gamma } \Theta $. 
\end{proof}

For the main analysis we need the following concepts of directional differentiability and local Lipschitz continuity.
\begin{definition}[(twice) directional differentiability, Definition 1.1.3 and (4.10) of \cite{CP-book}]  \label{def:1-dd}
	Given an open subset $ \mathcal{O} $ of $ \R^{n} $ and a scalar-valued function $ f : \mathcal{O}  \rightarrow \R $. The directional derivative of $ f $ at a point $ x\in    \mathcal{O}   $ along a direction $ d\in \R^{n} $ is defined as
	\begin{align}
		f^{\prime}(x;d):=\lim_{ \tau \downarrow 0}  \frac{f(x+\tau d ) - f(x) }{\tau}, \label{eq:1-dd}
	\end{align} if the limit exists.
	The function $ f $ is directionally differentiable at $ x $, if the limit \eqref{eq:1-dd} exists for all $ d \in \R^{n}$.
	The second-order directional derivative of $ f $ at a point $ x\in    \mathcal{O}   $ along a direction $ d\in \R^{n} $ is defined as
	\begin{align}
		f^{(2)} ( x; d ) := \lim_{ \tau \downarrow 0}  \frac{f(x+\tau d ) - f(x) - \tau f^{\prime}(x;d)  }{\tau^{2} /2 }, \label{eq:2-dd}
	\end{align}
	if the limit and the one for \eqref{eq:1-dd} exist. The function $ f $ is twice directionally differentiable at $ x $, if the limits \eqref{eq:1-dd} and \eqref{eq:2-dd} exist for all $ d \in \R^{n}$.
	
	For a vector-valued function $ f: \mathcal{O} \rightarrow \R^{m} $ with component functions $ \{ f_{i}: \mathcal{O} \rightarrow \R, i\in [m] \} $, the directional derivative $ f^{\prime}( x;d ) $ is defined as
	$$ f^{\prime}( x;d ):=  ( f_{1}^{\prime}( x;d ) , \dots, f_{m}^{\prime}( x;d ) )^{\top}, $$
	if $ f_{i}^{\prime}( x;d ) , i\in [m]$ exist.
	Furthermore, if $ f_{i}^{(2)}( x;d ) , i\in [m]$ exist, then the second-order directional derivative  $ f^{(2)}( x;d ) $ is defined as
	$$ f^{(2)}( x;d ):=  ( f_{1}^{(2)}( x;d ) , \dots, f_{m}^{(2)}( x;d ) )^{\top} . $$
	Function $f$ is (twice) directionally differentiable at $x$, if all of its component functions are (twice) directionally differentiable at $x$.
\end{definition}

\begin{definition}[local Lipschitz continuity]
	\label{def:lLc}
	For any function $ f : \mathcal{O} \subseteq \R^{n} \rightarrow \R^{m} $, we say $f$ is locally Lipschitz continuous near $ x \in \mathcal{O} $, if there exists a neighborhood $ X  $ of $x$ and $ K\geq 0 $ such that $ \| f(x_{1})-f(x_{2}) \| \leq K \|x_{1} - x_{2} \| $ for all $ x_{1} ,x_{2} \in X $.
	And we say $f$ is locally Lipschitz continuous, if $f$ is locally Lipschitz continuous near every point in its domain $ \mathcal{O} $.
\end{definition}

If $f$ is directionally differentiable at $x$ and locally Lipschitz continuous near $x$ with modulus $K \geq 0  $, then it follows from Definitions \ref{def:1-dd} and \ref{def:lLc} that for all $ d \in \R^{n} $,
\begin{align}
	\label{eq:mid2-0}
	\| f^{\prime}(x;d) \|
	= \left\|\lim_{\tau \downarrow 0 }  \frac{ f(x+\tau d) - f(x) }{ \tau } \right\|
	\leq \lim_{\tau \downarrow 0 } \frac{ K\tau \|d\| }{ \tau }
	= K \|d\|
	< \infty.
\end{align}
Hence, for all $ d , \bar{d} \in \R^{n} $, $\| f^{\prime}(x;d) - f^{\prime}(x; \bar{d}) \|$ is well-defined and we can similarly obtain that
\begin{align}
	\label{eq:mid2-1}
	\| f^{\prime}(x;d) - f^{\prime}(x; \bar{d}) \| \leq K \|d - \bar{d} \| , \,
	\forall d, \bar{d} \in \R^{n}.
\end{align}

The remaining analysis will be based on the following assumptions about the directional differentiability and local Lipschitz continuity.
\begin{assumption}\label{as1}
	Functions $g $ and $\{ \psi_{\ell-1} , \ell \in [L] \} $ are directionally differentiable on $ \R^{\bar{N}_L} $ and $ \R^{ n +  \bar{ N}_{\ell-1} },\ell \in [L] $ respectively.
	Functions $g $ and $\{ \psi_{\ell-1} , \ell \in [L] \} $ are locally Lipschitz continuous.
\end{assumption}
According to Lemma \ref{lem:bd-opt-1}, $ lev_{\leq \bar{\gamma} } \Theta $ with $ \bar{\gamma}:=\Theta(z^{0}) $ defined in \eqref{def:gammabar} is nonempty and compact.
Under Assumption \ref{as1}, it follows from the compactness of $ lev_{\leq \bar{\gamma} } \Theta $ that there exist $ K_g>0 $ and $ \{ K_{\ell}>0, \ell \in [L-1] \} $ such that
\begin{align}
	&| g( u ) - g( \bar{u} )  | \leq K_{g} \| u - \bar{u} \| ,  \label{def:Lipg} \\
	& \| \psi_{\ell-1}( \theta, \mathbf{u}_{\ell-1} ) - \psi_{\ell-1}(\theta, \bar{\mathbf{u}}_{\ell-1}  ) \|
	\leq  K_{\ell-1} \|
	\mathbf{u}_{\ell-1} - \bar{\mathbf{u}}_{\ell-1} \| , \, \ell=2,\dots,L \label{def:Lips}
\end{align}
for all $ (\theta^{\top} , \mathbf{u}^{\top})^{\top} $, $ (\theta^{\top} , \bar{\mathbf{u} }^{\top})^{\top}  \in \left( lev_{\leq \bar{\gamma} } \Theta + \epsilon \mathbb{B}(\mathbf{0};1) \right) $, where the positive real number $\epsilon$ is sufficiently small.
In \eqref{def:Lips}, the two terms at the left-hand side are consistent in the component $ \theta $, since the subsequent analysis only needs the Lipschitz continuity moduli of $ \{ \psi_{\ell}, \ell \in [L-1] \} $ in component $u$.
Besides, it should be noted that $ K_{g} $ and $ \{ K_{\ell},\ell\in[L-1] \} $ are non-increasing as $ \{ \beta_{\ell}, \ell \in [L] \}$ increase since functions $ g, \{ \psi_{\ell}, \ell \in [L-1] \} $ and $ \bar{\gamma}=F(z^{0}) $ are independent of penalty parameters and for any $z$, $ \Theta(z) $ is non-decreasing as $ \{ \beta_{\ell}, \ell \in [L] \} $ increase.
Furthermore, in Example \ref{ex2} and Section \ref{sec4} we will show how to estimate $K_{g}$ and $ \{ K_{\ell}, \ell\in [L-1] \} $ for specific applications.
\begin{remark}\label{remark:z0}
	In fact, $z^{0}$ and $ \bar{\gamma} $ can be replaced by any feasible point of \eqref{eq} and the value of $ \Theta $ at that point. The replacement will not affect any theoretical results in this paper.
	We choose $ z^{0} $ and $ \bar{\gamma} $ defined in \eqref{def:gammabar} since $ \mathbf{0} $ is usually not a good candidate  in data fitting. Thus, the requirement $ z\in lev_{\leq \bar{\gamma}} \Theta $ can be regarded as a mild condition.
\end{remark}

For clarity, break the direction $ d \in \R^{\bar{N}} $ according to the blocks of variable $z$ as follows
\begin{equation}
	\label{def:d}
	\begin{aligned}
		&d=(d_{\theta}^{\top}, d_{u}^{\top})^{\top}, \text{ where } \\
		&~d_{u}= (  ( d_{u_{1}  } )^{\top},\dots, ( d_{u_{L} } )^{\top}  )^{\top} \text{ with }
		d_{u_{\ell}  } \in \R^{N_{\ell}}, \ell\in[L].
	\end{aligned}
\end{equation}
Then under Assumption \ref{as1}, it follows from \eqref{eq:mid2-0}-\eqref{eq:mid2-1} and \eqref{def:Lipg}-\eqref{def:Lips} that for any $z  \in lev_{\leq \bar{\gamma} } \Theta $, and any $ d,\bar{d} \in \R^{\bar{N}} $ with $ d_{\theta} = \bar{d}_{\theta} $,
\begin{equation}
	\label{eq:Lip}
	\begin{aligned}
		&| g^{\prime}(u;d_{u}) - g^{\prime}(u;\bar{d}_{u}) | \leq K_{g} \|d_{u}-\bar{d}_{u}  \|, \\
		& \| \psi_{\ell-1}^{\prime}( \theta, \mathbf{u}_{\ell-1}
		; d_{\theta},  d_{ \mathbf{u}_{\ell-1} }  )
		- \psi_{\ell-1}^{\prime}( \theta, \mathbf{u}_{\ell-1} ; \bar{d}_{\theta},  \bar{d}_{ \mathbf{u}_{\ell-1} }  ) \|\\
		& ~ \leq K_{\ell-1} \|  d_{ \mathbf{u}_{\ell-1} } - \bar{d}_{ \mathbf{u}_{\ell-1} }     \|, \, \ell=2,\dots,L,
	\end{aligned}
\end{equation}
where $ d_{\mathbf{u}_{\ell-1}}:=(  ( d_{u_{1}  } )^{\top},\dots, ( d_{u_{\ell-1} } )^{\top}  )^{\top} $ for all $ \ell \in [L] $.  

Apart from Assumption \ref{as1}, the analysis concerning second-order necessary conditions also requires the following assumption about the twice directional differentiability.
\begin{assumption}\label{as2}
	Functions $g $ and $\{ \psi_{\ell-1} , \ell \in [L] \} $ are twice directionally differentiable on $ \R^{\bar{N}_L} $ and $ \R^{ n +  \bar{ N}_{\ell-1} },\ell \in [L] $, respectively.
\end{assumption}

We use the definitions of tangent cone \cite[Definition 6.1]{rockafellar2009variational} and radial cone \cite{PLQ_8}.
\begin{definition}[tangent cone and radial cone]
	\label{def:tanc}
	The tangent cone of a set $\mathcal{F} \subseteq \R^{m} $ at any point $ x \in \mathcal{F} $ is defined as
	\begin{align*}
		\mathcal{T}_{\mathcal{F}}(x)
		:= \{ d \in \R^{m} \mid    \exists x^{k} \rightarrow x \text{ with } x^{k} \in \mathcal{F}  \text{ and } \tau_{k} \downarrow 0, ~ \text{such that }   \frac{x^{k} - x}{\tau_{k}} \rightarrow d  \}.
	\end{align*}
	The radial cone of a set $\mathcal{F} \subseteq \R^{m} $ at any point $ x \in \mathcal{F} $ is defined as
	\begin{align*}
		P_{\mathcal{F}}(x):=\{ d \in  \R^{m}  \mid  \exists \tau_{k} \downarrow 0 \text{ such that }  x+ \tau_{k} d \in \mathcal{F}  \}.
	\end{align*}
\end{definition} 
Then it can be observed that $ P_{\mathcal{F}}(x) \subseteq \mathcal{T}_{\mathcal{F}}(x) $. When $ \mathcal{F} $ is convex, $P_{\mathcal{F}}(x) $ coincides with $\mathcal{T}_{\mathcal{F}}^{\circ }(x) $ used in \cite{min-maxR2}, and it further equals to $\mathcal{T}_{\mathcal{F}}(x)$ when $ \mathcal{F} $ is polyhedral.

\section{Optimality and stationarity}\label{sec3}
This section will establish the relationship between \eqref{eq:1.1}, \eqref{eq} and \eqref{l1pen} in global optimality and (second-order) d-stationarity, and discuss the byproducts regarding second-order sufficient conditions. The d-stationary points are defined by the necessary tangent condition outlined at the end of \cite[Chapter 8.C]{rockafellar2009variational} without proof. And the second-order d-stationarity extends the second-order necessary condition in \cite[Lemma 3.8]{min-maxR2} from a twice semidifferentiable objective function with convex constraints to a twice directionally differentiable objective function with general constraints. 
We provide a detailed proof for the necessity of (second-order) d-stationarity under the assumptions of the nonemptiness of solution sets, twice directional differentiability and local Lipschitz continuity of objective functions as follows.
\begin{lemma}\label{lem:1op}
	Assume that $ \argmin_{x \in \mathcal{F} } f(x) \neq \emptyset $ with $f: \R^{m} \rightarrow \R $.
	If $ \bar{x} \in \mathcal{F}  $ is a local minimizer of $ \min_{x \in \mathcal{F} } f(x) $, $f$ is directionally differentiable at $ \bar{x} $ and locally Lipschitz continuous near $ \bar{x} $, then $ f^{\prime}(\bar{x};d) \geq 0 $ for any $ d \in \mathcal{T}_{\mathcal{F} }(\bar{x})  $.
	Moreover, if $f$ is twice directionally differentiable at $ \bar{x} $, then $ f^{(2)}( \bar{x} ; d ) \geq 0 $ for all $ d \in P_{\mathcal{F}}  (\bar{x} ) \cap \{ d \in \R^{m} \mid f^{\prime}(\bar{x};d) =0 \} $.
\end{lemma}
\begin{proof}
	Firstly, the local optimality implies that $ \bar{x} \in \mathcal{F}  $ is a local minimizer of $ (f + \delta_{\mathcal{F} } ) (x)$ over $ \R^{m}  $.
	Hence it follows from  \cite[Theorem 10.1]{rockafellar2009variational} that
	\begin{align*}
		\liminf_{\tau \downarrow 0, d^{\prime} \rightarrow d } \frac{ (f+ \delta_{\mathcal{F}} )( \bar{x} + \tau d^{\prime} ) - (f+ \delta_{\mathcal{F}} )( \bar{x} ) }{ \tau } \geq 0, \, \forall d \in \R^{m}.
	\end{align*}
	Then the remainder is to show that for any $ d \in \mathcal{T}_{\mathcal{F} }(\bar{x}) $,
	\begin{align*}
		\liminf_{\tau \downarrow 0, d^{\prime} \rightarrow d } \frac{ (f+ \delta_{\mathcal{F}} )( \bar{x} + \tau d^{\prime} ) - (f+ \delta_{\mathcal{F}} )( \bar{x}   ) }{ \tau }  \leq f^{\prime}(\bar{x};d).
	\end{align*} 
	For any $ d \in \mathcal{T}_{\mathcal{F}}(\bar{x}) $, it follows from the definition of tangent cone that there exist $ x^{k}  \rightarrow \bar{x}, x^{k} \in  \mathcal{F} $ and $ \tau_{k} \downarrow 0 $ such that $ d^{k} :=\frac{x^{k} - \bar{x}}{\tau_{k}} \rightarrow d $ as $ k \rightarrow \infty $, which implies
	\begin{align*}
		&\liminf_{d^{\prime}  \rightarrow d , \tau \downarrow 0} \frac{ f (\bar{x} + \tau d^{\prime}) - f(\bar{x}) + \delta_{\mathcal{F}} (\bar{x} + \tau d^{\prime}  ) -  \delta_{\mathcal{F}}(\bar{x}) }{\tau}  \\
		\leq \, & \liminf_{k \rightarrow \infty } \frac{ f (\bar{x} + \tau_{k} d^{k}) - f(\bar{x}) + \delta_{\mathcal{F}} (\bar{x} + \tau_{k} d^{k}  ) -  \delta_{\mathcal{F}}(\bar{x}) }{\tau_{k}} \\
		= \, & \liminf_{k \rightarrow \infty } \frac{ f (\bar{x} + \tau_{k} d^{k}) - f(\bar{x})  }{\tau_{k}} \\
		= \, & f^{\prime} ( \bar{x} ; d ) ,
	\end{align*}
	where the first equality uses $ \bar{x} + \tau_{k} d^{k} = x^{k} \in \mathcal{F} $ and $ \bar{x} \in \mathcal{F} $, the last equality comes from
	\begin{align*}
		&\lim_{d^{\prime}  \rightarrow d , \tau \downarrow 0} \frac{ f (\bar{x} + \tau d^{\prime}) - f(\bar{x})  }{\tau} \notag \\
		=\,& \lim_{d^{\prime}  \rightarrow d , \tau \downarrow 0} \frac{ f (\bar{x} + \tau d ) - f(\bar{x})  }{\tau}
		+\lim_{d^{\prime}  \rightarrow d , \tau \downarrow 0} \frac{ f (\bar{x} + \tau d^{\prime}) - f (\bar{x} + \tau d )  }{\tau} \notag \\
		=\,& \lim_{ \tau \downarrow 0} \frac{ f (\bar{x} + \tau d ) - f(\bar{x})  }{\tau} +0  
		=  f^{\prime} ( \bar{x} ; d )
	\end{align*}
	by the fact that $f$ is directionally differentiable at $\bar{x}$ and locally Lipschitz continuous near $\bar{x}$.

	For the second-order optimality condition, since $ \bar{x} \in \mathcal{F} $ is a local minimizer of $\min_{x \in \R^{m} } (f + \delta_{\mathcal{F}} ) (x) $, it follows from \cite[Theorem 13.24 (a)]{rockafellar2009variational} that for all $ d \in \R^{m} $,
	\begin{align*}
		0 \leq \liminf_{ \tau \downarrow 0, d^{\prime} \rightarrow d  } \frac{(f+ \delta_{\mathcal{F}} )( \bar{x} + \tau d^{\prime} ) - (f+ \delta_{\mathcal{F}} )( \bar{x}   )}{ \tau^{2} /2 },
	\end{align*}
	which implies that for all $ d \in P_{\mathcal{F}} (\bar{x} ) \cap \{ d \in \R^{m} \mid f^{\prime}(\bar{x};d) =0 \} $,
	\begin{align*}
		0 & \leq \liminf_{ \tau \downarrow 0 } \frac{(f+ \delta_{\mathcal{F}} )( \bar{x} + \tau d  ) - (f+ \delta_{\mathcal{F}} )( \bar{x}   )}{ \tau^{2} /2 }
		\leq \liminf_{ k \rightarrow \infty } \frac{f( \bar{x} + \tau_{k} d  ) - f( \bar{x}   )}{ \tau_{k}^{2} /2 } \\
		& = \liminf_{ k \rightarrow \infty } \frac{f( \bar{x} + \tau_{k} d  ) - f( \bar{x}   ) - \tau_{k} f^{\prime}(\bar{x};d) }{ \tau_{k}^{2} /2 }
		= f^{(2) } (\bar{x};d),
	\end{align*}
	where the second inequality uses the definition of $ P_{\mathcal{F}} (\bar{x} ) $, the first equality holds due to $ f^{\prime}(\bar{x};d) =0 $, and the last equality comes from the twice directional differentiability of $ f $ at $ \bar{x} $ and $ \tau_{k} \downarrow 0 $. 
\end{proof}

The first-order condition for convexly constrained optimization problems with a semidifferentiable objective function is established in \cite[Lemma 3.8]{min-maxR2}. Lemma \ref{lem:1op} extends this condition to nonconvexly constrained optimization problems with a directionally differentiable objective function.
Actually, the two first-order conditions are completely identical in form since the directional derivative equals to the subderivative used in \cite[Lemma 3.8]{min-maxR2} under directional differentiability and local Lipschitz continuity.
However, for the second-order condition, the nonconvexity of the feasible region and non-twice-semidifferentiability of the objective function entail us to narrow the range of directions from $ \mathcal{T}^{\circ}_{\mathcal{F}}(x) $ in \cite{min-maxR2} to $ P_{\mathcal{F}}(x) $ and relax the nonnegativity of second-order subderivatives to the nonnegativity of second-order directional derivatives.
Based on Lemma \ref{lem:1op}, we can define unified first-order and second-order necessary conditions as follows.
\begin{definition}[second-order d-stationary point]
	\label{def:1op}
	For any $f: \R^{m} \rightarrow \R $ that is directionally differentiable on $ \R^{m} $ and locally Lipschitz continuous,  and any $\mathcal{F}\subseteq \R^{m}$ such that $ \emptyset \neq \argmin_{ x \in \mathcal{F} } f(x) $, we call $ \bar{x} \in \mathcal{F}  $ a d(irectional)-stationary point of $ \min_{x\in\mathcal{F}} f(x) $ if $  f^{\prime} ( \bar{x} ; d ) \geq 0 $ for all $ d \in \mathcal{T}_{\mathcal{F} }(\bar{x})  $.
	And we further call $ \bar{x}  $ a second-order d-stationary point of $ \min_{x\in\mathcal{F}} f(x) $ if $f$ is also twice directionally differentiable at $ \bar{x} $ with $ f^{(2)}( \bar{x} ; d ) \geq 0 $ for all $ d \in P_{\mathcal{F}}  (\bar{x} ) \cap \{ d \in \R^{m} \mid f^{\prime}(\bar{x};d) =0 \} $.
\end{definition}

\begin{example} To illustrate Definition \ref{def:1op}, we consider
	$ \min_{x \in \mathcal{F} } f(x)$ with $f(x)= \max \{ -1, x_{1} x_{2} \} + 0.1 \|x\|^{2}$ and $ \mathcal{F}=[-1,1]^{2} $.
	This example has only three first-order d-stationary points $(0,0)^\top, (-1,1)^\top $ and $ (1,-1)^\top $,  while $f$ is differentiable at the first point, but not differentiable at the other two points.
	At $ \bar{x}:=(0,0)^\top$, $ f^{\prime}(\bar{x};d )=\nabla f(\bar{x})^\top d \equiv 0 $ for all $ d \in \mathcal{T}_{\mathcal{F}}(\bar{x})=\R^{2} $, whereas $ f^{(2)}( \bar{x} ; d )= 2 d_{1} d_{2} + 0.2 \| d \|^{2} <0 $ for any $ d:=(d_1,d_2)^{\top} \in  \{ d \in P_{\mathcal{F}} (\bar{x})  \mid f^{\prime}(\bar{x};d ) =0 \} =  \R^{2} $ with $d_2=-d_1 \neq 0 $. Hence
	$(0,0)^\top$ is not a second-order d-stationary point, and thus not a local minimizer.
	
	In contrast, at $ \bar{x}:=(-1,1)^\top$,   $ f^{\prime}(\bar{x};d )= (d_{1} - d_{2}) - 0.2 ( d_{1} - d_{2} ) \geq 0 $  for all $ d \in \mathcal{T}_{\mathcal{F}}(\bar{x})=\{ (d_{1}, d_{2} )^{\top} \mid d_{1} \geq 0, d_{2} \leq 0 \} $, and $ f^{(2)}( \bar{x} ; d )= 0 \geq 0  $ for any $ d \in  \{ d \in P_{\mathcal{F}} (\bar{x})  \mid f^{\prime}(\bar{x};d ) =0 \}= \{ d \in\mathcal{T}_{\mathcal{F}}(\bar{x}) \mid f^{\prime}(\bar{x};d ) =0 \} =  \{ \mathbf{0} \} $. Hence $(-1,1)^\top$ is a second-order d-stationary point. Similarly, we can verify that $(1,-1)^\top$ is also a second-order d-stationary point. From the boundedness of the feasible set  ${\cal F}$ and the continuity of the objective function $f$, the optimal solution set of this example is nonempty.
	Since $f((-1,1)^\top)=f((1,-1)^\top)$, the two points are optimal solutions.
\end{example}

Since the local Lipschitz continuity of $ \Psi $, $ F $ and $ \Theta $ naturally holds under Assumption \ref{as1}, the corresponding d-stationary points can be defined once we have checked their directional differentiability.
\begin{proposition}\label{dd_F1}
	Under Assumption \ref{as1}, $ \Psi $ is directionally differentiable on $ \R^{n} $, and the directional derivative of the objective function of  \eqref{eq:1.1} along any $ d_{\theta} \in \R^{n} $ is
	\begin{align}\label{dd}
		g^{\prime} \left( u_{1}, \dots, u_{L} ;  d_{u_{1}}  , \dots ,    d_{u_{L}} \right)
		+   2 \lambda \theta^{\top} d_{\theta} ,
	\end{align}
	where $ u_{\ell}:= \psi_{\ell-1}( \theta, u_{1}, \dots, u_{\ell-1} ) $ for all $ \ell \in [L] $, and $ d_{u_{\ell}} := \psi_{\ell-1}^{\prime} ( \theta, u_{1},\dots,u_{\ell-1};  d_{\theta}, d_{u_{1}} , \dots, d_{u_{\ell-1}} )$ for all $ \ell \in [L] $; 
	$ F $ and $ \Theta $ are directionally differentiable on $ \R^{\bar{N}} $, and for any direction $d \in \R^{\bar{N}} $ defined in \eqref{def:d},
	\begin{align}
		F^{\prime}(z;d)=g^{\prime}(u;d_{u}) + 2 \lambda \theta^{\top} d_{\theta} , \label{dd0}
	\end{align}
	\begin{align}
		\Theta^{\prime} (z;d)
		=\,& F^{\prime}(z;d)
		+ \sum_{\ell=1}^{L} \beta_{\ell} \left(
		\sum_{i \in I_{+}^{\ell}(z)  } [ d_{u_{\ell}} - \psi_{\ell-1}^{\prime} ( \theta, \mathbf{u}_{\ell-1} ; d_{\theta},   d_{\mathbf{u}_{\ell-1}} )  ]_{i} \right. \label{dd1} \\
		&\left. - \sum_{i \in I_{-}^{\ell}(z)  } [ d_{u_{\ell}} - \psi_{\ell-1}^{\prime} ( \theta, \mathbf{u}_{\ell-1} ; d_{\theta},   d_{\mathbf{u}_{\ell-1}} )  ]_{i} \right. \notag \\
		&\left. +\sum_{i \in I_{0}^{\ell}(z)  } \left| [ d_{u_{\ell}} - \psi_{\ell-1}^{\prime} ( \theta, \mathbf{u}_{\ell-1} ; d_{\theta},  d_{\mathbf{u}_{\ell-1}} )  ]_{i} \right|
		\right)  , \notag
	\end{align}
	where for all $ \ell \in [L] $,
	\begin{align*}
		&
		I_{+}^{\ell}(z):=\{ i \in [ N_{\ell}] \mid [ u_{\ell} - \psi_{\ell-1}( \theta, \mathbf{u}_{\ell-1} ) ]_{i}  >0 \}, \\
		&
		I_{-}^{\ell}(z):=\{ i \in [ N_{\ell}] \mid [ u_{\ell} - \psi_{\ell-1}( \theta, \mathbf{u}_{\ell-1} ) ]_{i}  <0 \} ,  \\
		&
		I_{0}^{\ell}(z):= [N_{\ell}] \backslash \left( I_{+}^{\ell}(z)  \cup I_{-}^{\ell}(z) \right) .
	\end{align*}
\end{proposition}
\begin{proof}
	Firstly, applying \cite[Proposition 4.1.2]{CP-book} sequentially on $ \psi_{1}(\cdot, \psi_{0}(\cdot)),  \dots ,   \psi_{L-1}(\cdot, \psi_{0}(\cdot), \psi_{1}(\cdot, \psi_{0}(\cdot)), \cdots ) $ and $ \Psi $ with directionally differentiable and locally Lipschitz continuous $ \{ \psi_{\ell-1} , \ell \in [L] \} $ and $g$, we can obtain
	\begin{align*}
		\Psi^{\prime}( \theta; d_{\theta} )=&  g^{\prime} \left( \psi_{0}(\theta ) , ~~~~\psi_{1}( \theta, \psi_{0}(\theta) ) , \dots ; \right.
		\\
		&~~~\left. \psi_{0}^{\prime}(\theta; d_{\theta} ),  \psi_{1}^{\prime}( \theta, \psi_{0}(\theta) ; d_{\theta} , \psi_{0}^{\prime}(\theta; d_{\theta} ) )  , \dots \right),
	\end{align*}
	which can be reorganized as \eqref{dd} with $ \{ u_{\ell}, d_{ u_{\ell} } , \ell \in [L] \} $ defined as above.
	Then the result about $F$ can be directly obtained from Assumption \ref{as1} and Definition \ref{def:1-dd}.
	And for $\Theta$, it is sufficient to show the directional differentiability of each penalty term according to Definition \ref{def:1-dd}, which can be obtained by the directional differentiability and local Lipschitz continuity of $ \|\cdot\|_{1} $ and $ \{ \psi_{\ell-1} , \ell \in [L] \} $ \cite[Proposition 4.1.2]{CP-book}. 
\end{proof}

Together with the nonemptiness of optimal solution sets $ \mathcal{S}, \,   \mathcal{S}_{0} $ and $\mathcal{S}_{1}   $ and Definition \ref{def:1op}, it implies that under Assumption \ref{as1},
\begin{itemize}
	\item $ \mathcal{D}:= \{ \theta \in \R^{n} \mid [ \Psi(\cdot) + \lambda\|\cdot\|^2 ]^{\prime}(\theta;d_{\theta}) \geq 0  \text{ for all } d_{\theta}  \in \R^{n} \} $ is the set of d-stationary points of \eqref{eq:1.1};
	\item $ \mathcal{D}_{0}:= \{ z\in \R^{\bar{N}} \mid z \in \mathcal{F}_{0} \text{ and } F^{\prime} (z  ; d ) \geq 0 \text{ for all } d \in \mathcal{T}_{\mathcal{F}_{0}}(z)  \}  $ is the set of d-stationary points of \eqref{eq};
	\item $ \mathcal{D}_{1}:= \{ z\in \R^{\bar{N}} \mid \Theta^{\prime}(z;d) \geq 0 \text{ for all } d \in \R^{\bar{N}}  \}  $ is the set of d-stationary points of \eqref{l1pen}.
\end{itemize}
Although the d-stationary point of \eqref{eq:1.1} can be defined as above, the complicated nested structure in $ \Psi^{\prime} $ makes it challenging to compute. Notably, through \eqref{eq}, we can express its d-stationarity more clearly and further establish \eqref{eq:1.1}'s relationship with \eqref{l1pen} in Sections \ref{sec3.1} and \ref{sec3.2}, which facilitates computation.

And for the second-order necessary conditions of \eqref{eq} and \eqref{l1pen}, the twice directional differentiability of $   F $ and $ \Theta $ can be verified under Assumptions \ref{as1} and \ref{as2}.
\begin{proposition}\label{2dd_F1}
	Under Assumptions \ref{as1} and \ref{as2}, $F $ and $ \Theta $ are twice directionally differentiable on $ \R^{\bar{N}} $ and for any direction $ d \in \R^{\bar{N}} $ defined in \eqref{def:d},
	\begin{align*}
		F^{(2)}(z;d)=\,&g^{(2)}(u;d_{u}) + 2 \lambda  \| d_{\theta} \|^{2} , 
		\\ 
		\Theta^{(2)} (z;d)
		=\,& F^{(2)}(z;d)  
		+ \sum_{\ell=1}^{L} \beta_{\ell} \left(
		- \sum_{i \in I_{+}^{\ell}(z) \cup I_{0,+}^{\ell}(z;d)  } [  \psi_{\ell-1}^{(2)} ( \theta, \mathbf{u}_{\ell-1} ; d_{\theta},  d_{\mathbf{u}_{\ell-1}} )  ]_{i} \right. \notag \\
		&\left. + \sum_{i \in I_{-}^{\ell}(z) \cup I_{0,-}^{\ell}(z;d) } [   \psi_{\ell-1}^{(2)} ( \theta, \mathbf{u}_{\ell-1} ; d_{\theta},   d_{\mathbf{u}_{\ell-1}} )  ]_{i} \right. \notag \\
		&\left. +\sum_{i \in I_{0,0}^{\ell}(z;d)  } \left| [ \psi_{\ell-1}^{(2)} ( \theta, \mathbf{u}_{\ell-1} ; d_{\theta},   d_{\mathbf{u}_{\ell-1}} )  ]_{i} \right|
		\right)  , \notag
	\end{align*}
	where for all $ \ell \in [L] $, $ I_{+}^{\ell}(z), I_{-}^{\ell}(z) $ and $ I_{0}^{\ell}(z) $ are defined as in Proposition \ref{dd_F1} and
	\begin{align*}
		&
		I_{0,+}^{\ell}(z;d):= \{ i \in I_{0}^{\ell}(z) \mid [ d_{u_{\ell}} - \psi_{\ell-1}^{\prime} ( \theta, \mathbf{u}_{\ell-1} ; d_{\theta},  d_{\mathbf{u}_{\ell-1}} )  ]_{i} >0 \} , \\
		&
		I_{0,-}^{\ell}(z;d):= \{ i \in I_{0}^{\ell}(z) \mid [ d_{u_{\ell}} - \psi_{\ell-1}^{\prime} ( \theta, \mathbf{u}_{\ell-1} ; d_{\theta},   d_{\mathbf{u}_{\ell-1}} )  ]_{i} <0 \} ,  \\
		&
		I_{0,0}^{\ell}(z;d):= I_{0}^{\ell}(z) \backslash \left( I_{0,+}^{\ell}(z;d)  \cup I_{0,-}^{\ell}(z;d) \right) .
	\end{align*}
\end{proposition}
\begin{proof}
	The result about $F$ can be directly obtained from Assumptions \ref{as1}, \ref{as2} and Definition \ref{def:1-dd}.
	And for $\Theta$, it is sufficient to show the twice directional differentiability of each penalty term according to Definition \ref{def:1-dd}, which can be obtained by the twice directional differentiability and local Lipschitz continuity of $ \{ \psi_{\ell-1} , \ell \in [L] \} $ and the max-structure of  $ |\cdot |   $ (i.e., for any $ x\in \R, |x|= \max\{ x, -x \} $) \cite[Example 4.2.1]{CP-book}. 
\end{proof}
Together with Proposition \ref{dd_F1}, it implies that under Assumptions \ref{as1} and \ref{as2},
\begin{itemize}
	\item $ \mathcal{SD}_{0}:= \{ z \in \mathcal{D}_{0} \mid F^{(2)}(z  ; d ) \geq 0 \text{ for all } d \in P_{\mathcal{F}_{0}} (z ) \text{ satisfying } F^{\prime} (z  ; d ) = 0\}$ is the set of second-order d-stationary points of \eqref{eq};
	\item $ \mathcal{SD}_{1}:= \{ z \in \mathcal{D}_{1} \mid \Theta^{(2)}( z ; d )\geq 0 \text{ for all } d \in \R^{\bar{N}} \text{ satisfying } \Theta^{\prime} (z  ; d ) = 0 \}$ is the set of second-order d-stationary points of \eqref{l1pen}.
\end{itemize}
Although it seems immature to define the second-order d-stationary point of \eqref{eq:1.1} under Assumptions \ref{as1} and \ref{as2}, we could next see how the sets $ \mathcal{SD}_{0},\mathcal{SD}_{1} $ help to provide second-order necessary conditions for \eqref{eq:1.1}.

\subsection{Closed-form of $ \mathcal{T}_{\mathcal{F}_{0}} $}\label{sec3.1}
Here we give the closed-form of the tangent cone of $ \mathcal{F}_{0} $ in \eqref{def:set-F0} based on directional derivatives of constraints of \eqref{eq}, which plays an important role in the following subsections.

\begin{theorem} \label{th:T-express-2}
	Under Assumption \ref{as1}, it holds that
	\begin{align*}
		\mathcal{T}_{\mathcal{F}_{0}}(z)
		= \{ d \in \R^{\bar{N}} \mid  d_{u_{\ell}} = \psi_{\ell-1}^{\prime} ( \theta,  \mathbf{u}_{\ell-1}; d_{\theta},   d_{\mathbf{u}_{\ell-1}} ), \ell\in[L]  \},
	\end{align*}
	for any feasible point $ z \in \mathcal{F}_{0} $.
\end{theorem}
\begin{proof}
	We firstly prove the one-sided inclusion
	\begin{align*}
		\mathcal{T}_{\mathcal{F}_{0}}(z)
		\subseteq \{ d \in \R^{\bar{N}} \mid  d_{u_{\ell}} = \psi_{\ell-1}^{\prime} ( \theta,  \mathbf{u}_{\ell-1} ; d_{\theta},  d_{\mathbf{u}_{\ell-1}}  ), \ell\in[L]  \}.
	\end{align*}
	According to Definition \ref{def:tanc}, for any $ d \in \mathcal{T}_{\mathcal{F}_{0}}(z) $, there exists a sequence $ \{ z^{k} \in \mathcal{F}_{0} , k \in \mathbb{Z}_{++} \} $ converging to $z$ and a sequence $ \tau_{k} \downarrow 0 $ such that
	$
	d= \lim_{k \rightarrow \infty } \frac{z^{k} -z }{\tau_{k}}
	$.
	Denoting $ d^{k} := \frac{z^{k} -z }{\tau_{k}} $, we have $ z^{k}=z+\tau_{k} d^{k}$ for all $k$ and $ d^{k} \rightarrow d $ as $ k \to \infty $. For $ \ell=1 $, it follows from the definition of $ \psi_{0}^{\prime} $, $ \tau_{k} \downarrow 0 $, local Lipschitz continuity of $\psi_{0}$ and $ d_{\theta}^{k} \rightarrow d_{\theta} $ that
	\begin{align*}
		d_{u_{1}} - \psi_{0}^{\prime} ( \theta; d_{\theta} ) 
		= \, & \lim_{k \rightarrow \infty } \left[ \frac{ u_{1}^{k} - u_{1} }{ \tau_{k} }  - \frac{ \psi_{0}( \theta+\tau_{k} d_{\theta}^{k} ) - \psi_{0}( \theta ) }{\tau_{k}}\right] \\
		=\, & \lim_{k \rightarrow \infty }   \frac{ [u_{1}^{k} - \psi_{0}( \theta^{k} ) ] -[ u_{1} - \psi_{0}( \theta )] }{ \tau_{k} }
		=0,
	\end{align*}
	where 
	the second equality uses $ \theta^{k} = \theta+\tau_{k} d_{\theta}^{k} $, the last equality comes from $ z^{k},z \in \mathcal{F}_{0} $.
	For $\ell=2,\dots,L$,
	it follows from the definition of $ \psi_{\ell-1}^{\prime} $, $ \tau_{k} \downarrow 0 $, local Lipschitz continuity of $\psi_{\ell-1}$ and $ (d_{\theta}^{k}, d_{\mathbf{u}_{\ell-1}}^{k} ) \rightarrow (d_{\theta}, d_{\mathbf{u}_{\ell-1}}  ) $ that
	\begin{align*}
		& d_{u_{\ell}} - \psi_{\ell-1}^{\prime} ( \theta,  \mathbf{u}_{\ell-1}  ; d_{\theta}, d_{\mathbf{u}_{\ell-1}} ) \\ 
		= \, & \lim_{k \rightarrow \infty } \left[ \frac{ u_{\ell}^{k} - u_{\ell} }{ \tau_{k} }
		- \frac{ \psi_{\ell-1}( \theta+\tau_{k} d_{\theta}^{k},  \mathbf{u}_{\ell-1}  + \tau_{k}  d_{\mathbf{u}_{\ell-1}}^{k} ) - \psi_{\ell-1}( \theta, \mathbf{u}_{\ell-1}  ) }{\tau_{k}}\right]\\
		=\, & \lim_{k \rightarrow \infty }   \frac{ [u_{\ell}^{k} - \psi_{\ell-1}( \theta^{k},  \mathbf{u}_{\ell-1}^{k}  ) ] -[ u_{\ell} - \psi_{\ell-1}( \theta ,  \mathbf{u}_{\ell-1}  )] }{ \tau_{k} }
		=0,
	\end{align*}
	where 
	the second equality holds due to $ z^{k} = z+\tau_{k} d^{k} $, the last equality comes from $ z^{k},z \in \mathcal{F}_{0} $.

	Next we deduce the reverse inclusion
	\begin{align*}
		\mathcal{T}_{\mathcal{F}_{0}}(z)
		\supseteq \{ d \in \R^{\bar{N}} \mid  d_{u_{\ell}} = \psi_{\ell-1}^{\prime} ( \theta,  \mathbf{u}_{\ell-1}  ; d_{\theta}, d_{\mathbf{u}_{\ell-1}} ), \ell\in[L]  \}.
	\end{align*}
	By Definition \ref{def:tanc}, it is equivalent to show, for any $ d \in \R^{\bar{N}} $ satisfying
	\begin{equation}\label{eq:T2-1}
		\begin{aligned}
			d_{u_{\ell}} = \psi_{\ell-1}^{\prime} ( \theta,  \mathbf{u}_{\ell-1}  ; d_{\theta}, d_{\mathbf{u}_{\ell-1}} ), \,  \ell\in[L]  ,
		\end{aligned}
	\end{equation}
	there exist sequences $ \{ \tau_{k} \downarrow 0 \} $ and $ \{ d^{k} \rightarrow d \} $ such that $ \{  z+ \tau_{k} d^{k} \} \subseteq \mathcal{F}_{0} $.
	For any $d$ satisfying \eqref{eq:T2-1} and any decreasing sequence $ \{ \tau_{k} \downarrow 0 \} $, define
	\begin{align*}
		&d_{\theta}^{k} := d_{\theta}, ~
		d_{u_{1}}^{k} :=\frac{ \psi_{0}( \theta + \tau_{k} d_{\theta}^{k}  )  - \psi_{0}( \theta) }{ \tau_{k} }, \text{ and }\\
		&d_{u_{\ell}}^{k} :=\frac{ \psi_{\ell-1}( \theta + \tau_{k} d_{\theta}^{k},   \mathbf{u}_{\ell-1} + \tau_{k} d_{\mathbf{u}_{\ell-1}}^{k}  )  - \psi_{\ell-1}( \theta,  \mathbf{u}_{\ell-1} ) }{ \tau_{k} } \text{ in the order of }  \ell=2,\dots,L,
	\end{align*}
	for all $k  $.
	We prove that $ d^{k}\rightarrow d $ and $   z^{k}:= ( z+ \tau_{k} d^{k} ) \in \mathcal{F}_{0} $ for all $k  $. 
	Firstly, $ \lim_{k\rightarrow \infty} d_{\theta}^{k}=d_{\theta} $. Then it follows from the definition of $ d_{u_{1}}^{k} $, $ d_{\theta}^{k} =d_{\theta} $, $ \tau_{k} \downarrow 0 $ and the directional differentiability of $ \psi_{0} $ that
	$
	\lim_{k\rightarrow \infty }  d_{u_{1}}^{k} 
	= \psi_{0}^{\prime} ( \theta; d_{\theta} )
	= d_{u_{1}}
	$,
	where 
	the last equality holds due to \eqref{eq:T2-1}.
	For any $ \ell=2,\dots,L $, assume that $ d_{\mathbf{u}_{\ell-1}}^{k} \rightarrow d_{\mathbf{u}_{\ell-1}} $ has been verified. Then it follows from the definition of $ d_{u_{\ell}}^{k} $ and local Lipschitz continuity of $ \psi_{\ell-1} $ that
	\begin{align*}
		\lim_{k\rightarrow \infty }  d_{u_{\ell}}^{k} 
		= \, &
		\lim_{k\rightarrow \infty }  \frac{ \psi_{\ell-1}( \theta + \tau_{k} d_{\theta} ,   \mathbf{u}_{\ell-1} + \tau_{k} d_{\mathbf{u}_{\ell-1}}      )  - \psi_{\ell-1}( \theta,   \mathbf{u}_{\ell-1}    ) }{ \tau_{k} }\\
		= \,& \psi_{\ell-1}^{\prime} ( \theta,   \mathbf{u}_{\ell-1}     ; d_{\theta},   d_{\mathbf{u}_{\ell-1}}      )  
		= \,  d_{u_{\ell}},
	\end{align*}
	where the second equality comes from $ \tau_{k} \downarrow 0 $ and directional differentiability of $ \psi_{\ell-1} $, the last equality holds due to \eqref{eq:T2-1}. By induction and \eqref{def:d}, we obtain that $ d^{k} \rightarrow d $. 
	From $z^{k}:= ( z+ \tau_{k} d^{k} )$ and \eqref{def:z}, \eqref{def:d}, we first have
	\begin{align}\label{eq:T2-2}
		\theta^{k}:= \theta + \tau_{k} d_{\theta}^{k} ,\text{ and }
		u_{\ell}^{k} = u_{\ell}  + \tau_{k} d_{u_{\ell} }^{k}, \, \forall \ell\in [L].
	\end{align}
	For $\ell=1$, \eqref{eq:T2-2} and the definition of $ d_{u_{1} }^{k} $ imply that
	$
	u_{1}^{k}
	= u_{1}  + \psi_{0}( \theta^{k}   )  - \psi_{0}( \theta)
	= \psi_{0}( \theta^{k}   )
	$,
	where the last equality follows from $ z \in \mathcal{F}_{0} $.
	For $ \ell=2,\dots,L $, \eqref{eq:T2-2} and the definition of $ d_{u_{\ell} }^{k} $ imply that
	$
	u_{\ell}^{k} 
	= \,    u_{\ell}  +  \psi_{\ell-1}( \theta^{k} ,  \mathbf{u}_{\ell-1}^{k}    )  - \psi_{\ell-1}( \theta,  \mathbf{u}_{\ell-1}   )
	= \,  \psi_{\ell-1}( \theta^{k} ,  \mathbf{u}_{\ell-1}^{k}    )
	$,
	where the last equality follows from $ z \in \mathcal{F}_{0} $. Combining the results for all $ \ell \in [L] $, we obtain that $ z^{k}\in \mathcal{F}_{0} $ for all $k$. 
\end{proof}

It is noteworthy that Theorem \ref{th:T-express-2} provides the expression for $ \mathcal{T}_{\mathcal{F}_{0}}(z) $, where $\mathcal{F}_{0}$ is defined by nonsmooth constraints. In general, such an expression for the tangent cone of the feasible region is only achievable for smooth constraints under the linear independence constraint qualification (LICQ) \cite[Lemma 12.2]{nocedal2006numerical}.
The one-sided inclusion for $\mathcal{T}_{\mathcal{F}_{0}}(z) $ can only guarantee one-sided implication between \eqref{eq} and \eqref{l1pen} \cite[Theorem 9.2.1 and Remark 9.2.1]{CP-book}, while the closed-form in Theorem \ref{th:T-express-2} can guarantee the equivalence (see Theorem \ref{th:P0=P1}).
If we use constraint qualifications for nonsmooth constraints named NNAMCQ \cite{LReLU-LLC} and relations between tangent and normal cones \cite[Theorems 6.26 and 6.28]{rockafellar2009variational}, we can only obtain a subset of $ \mathcal{T}_{\mathcal{F}_{0}}(z) $
presented by $ \{ \psi_{\ell-1}^{\prime}, \ell\in[L] \} $, which fails to imply the full characterization of $ \mathcal{T}_{\mathcal{F}_{0}}(z) $ in certain cases such as $ \mathcal{T}_{\mathcal{F}_{0}^{RNN}}(z) $ in Section \ref{sec4}. 
Noting that $ P_{\mathcal{F}_{0}}(z)  \subseteq \mathcal{T}_{\mathcal{F}_{0}} (z)  $ for any $ z \in \mathcal{F}_{0} $, Theorem \ref{th:T-express-2} also provides the expression of a superset of $ P_{\mathcal{F}_{0}} (z) $, which helps to obtain its closed-form in certain cases (see Section \ref{sec4}).

\subsection{Equivalence in optimality and d-stationarity} \label{sec3.2}
Here we show the equivalence of \eqref{eq:1.1}, \eqref{eq} and \eqref{l1pen} in global optimality and d-stationarity.
Firstly, \eqref{eq:1.1} and \eqref{eq} are equivalent in global optimality as we discussed after Lemma \ref{lem:bd-opt-0}.
Similarly, according to Proposition \ref{dd_F1} and Theorem \ref{th:T-express-2}, \eqref{eq:1.1} and \eqref{eq} are equivalent in d-stationarity when neglecting dimension lifting.
\begin{lemma}\label{lem:P-P0}
	If $ \theta \in \mathcal{D} $, then $ z := ( \theta^{\top}, u_{1}^{\top}, \dots , u_{L}^{\top} )^{\top}  \in \mathcal{D}_{0}  $ where $u_{\ell}:= \psi_{\ell-1} ( \theta ,  u_{1},   \dots , u_{\ell-1} ) $ for all $ \ell \in[L]$.
	Conversely, if $ z:= ( \theta^{\top}, u^{\top} )^{\top}  \in \mathcal{D}_{0}  $, then $ \theta \in \mathcal{D} $.
\end{lemma}
\begin{proof}
	For any $ \theta \in \mathcal{D} $, it follows from $ z := ( \theta^{\top}, u_{1}^{\top}, \dots , u_{L}^{\top} )^{\top}  $ with $u_{\ell}:= \psi_{\ell-1} ( \theta ,   u_{1}, \dots , u_{\ell-1} ) $ for all $ \ell\in[L]$ that $ z \in \mathcal{F}_{0} $.
	Then for any $ d \in \mathcal{T}_{\mathcal{F}_{0}  }(z) $, we have $ d_{u_{\ell}} = \psi_{\ell-1}^{\prime} ( \theta,  \mathbf{u}_{\ell-1}; d_{\theta},   d_{\mathbf{u}_{\ell-1}} )$ for all $\ell \in[L] $ by Theorem \ref{th:T-express-2}.
	Together with Proposition \ref{dd_F1}, it implies that for any $ d \in \mathcal{T}_{\mathcal{F}_{0}  }(z) $,
	\begin{align*}
		F^{\prime}(z;d) = \Psi^{\prime}( \theta ; d_{\theta} ) + 2 \lambda \theta^{\top} d_{\theta} \geq 0,
	\end{align*}
	where the inequality comes from $ \theta \in \mathcal{D} $. 
	On the other hand, if $ z:= ( \theta^{\top}, u^{\top} )^{\top}  \in \mathcal{D}_{0}  $, then $ z \in \mathcal{F}_{0} $, i.e., $u_{\ell}:= \psi_{\ell-1} ( \theta ,  u_{1}, \dots , u_{\ell-1} ) $ for all $ \ell\in[L]$.
	Then for any $ d_{\theta} \in \R^{n} $, it follows from Proposition \ref{dd_F1} that under the setting of $ d:= ( d_{\theta}^{\top} , d_{u_{1}}^{\top}, \dots, d_{u_{L}}^{\top} )^{\top} $ with $ d_{u_{\ell}} := \psi_{\ell-1}^{\prime} ( \theta, u_{1},\dots,u_{\ell-1}; d_{\theta}, d_{u_{1}} , \dots, d_{u_{\ell-1}} )$ for all $ \ell \in [L] $,
	\begin{align*}
		\Psi^{\prime}( \theta ; d_{\theta} ) + 2 \lambda \theta^{\top} d_{\theta}
		=F^{\prime}(z;d)
		\geq 0,
	\end{align*}
	where the inequality uses $ d \in \mathcal{T}_{\mathcal{F}_{0}  }(z) $ from Theorem \ref{th:T-express-2} and $ z \in \mathcal{D}_{0}  $.
\end{proof} 

To establish the equivalence between \eqref{eq} and \eqref{l1pen}, inspired by Theorem 2.1 (a) of \cite{DNN-CHP}, we first show that under proper setting of $ \{ \beta_{\ell}>0, \ell\in [L] \} $ restricted by $ K_{g} $ and $ \{ K_{\ell}>0, \ell\in [L-1] \} $, the d-stationary point of \eqref{l1pen} in $ lev_{\leq \bar{\gamma}} \Theta $ must be feasible to \eqref{eq}, where $\bar{\gamma}$ is defined in \eqref{def:gammabar}.
\begin{lemma}\label{lem:D1-F0}
	Under Assumption \ref{as1},
	let $z$ be a d-stationary point of \eqref{l1pen} with $ \Theta(z) \leq \bar{\gamma} $ and $ \{ \beta_{\ell}>0, \ell\in [L] \} $ satisfying
	\begin{align} \label{eq:br-threshold-2}
		\beta_{\ell} >  K_{g}  \prod_{j=\ell+1}^{L} ( 1+ K_{j-1} ) , \text{ for all } \ell \in [L].
	\end{align}
	Then $ z \in \mathcal{F}_{0} $.
\end{lemma}
\begin{proof}
	To show $ z \in \mathcal{F}_{0} $, we are going to prove $ u_{\ell} = \psi_{\ell-1} ( \theta, u_{1}, \dots, u_{\ell-1} ) $ in the order of $ \ell=L,\dots,1 $ separately.

	To show $ u_{L} = \psi_{L-1} ( \theta, u_{1}, \dots, u_{L-1} ) $ by contradiction, we use $ I_{+}^{L}(z) , I_{-}^{L}(z) $ and $I_{0}^{L}(z)$ defined in Proposition \ref{dd_F1}.
	Suppose $ u_{L} \neq \psi_{L-1} ( \theta, u_{1}, \dots, u_{L-1} ) $, then $ I_{+}^{L}(z) \cup I_{-}^{L}(z)  \neq \emptyset $. Let $ \bar{z}:=( \bar{\theta}^{\top},  \bar{u}_{1}^{\top}, \dots, \bar{u}_{L}^{\top} )^{\top} $ with
	\begin{align*}
		&\bar{\theta}:=\theta, ~
		\bar{u}_{1} := u_{1}, ~ \dots, ~ \bar{u}_{L-1} := u_{L-1}, \\
		& \bar{u}_{L} : = \psi_{L-1} ( \theta, u_{1}, \dots, u_{L-1} ) \neq u_{L},
	\end{align*}
	and $ d:= \bar{z} -z= ( \mathbf{0}, \dots, \mathbf{0}, d_{u_{L}} ) \neq \mathbf{0} $. Then it follows from Proposition \ref{dd_F1} that
	\begin{align}
		\Theta^{\prime}( z;d )
		&= F^{\prime}(u;d_{u} ) + \beta_{L} \left(
		\sum_{i \in I_{+}^{L}(z) } [ d_{u_{L}} ]_{i}  -  \sum_{i \in I_{-}^{L}(z) } [ d_{u_{L}} ]_{i} + \sum_{i \in I_{0}^{L}(z) } | [ d_{u_{L}} ]_{i} |
		\right) \notag \\
		&= g^{\prime}(u;d_{u} ) - \beta_{L}  \| d_{u_{L}} \|_{1} \label{eq:4-0.5} \\
		&\leq ( K_{g}  - \beta_{L} ) \| d_{u_{L}} \|_{1}< 0, \notag
	\end{align}
	where the two equalities use the definitions of $d$ and $ I_{+}^{L}(z), I_{-}^{L}(z), I_{0}^{L}(z)  $, the first inequality holds due to \eqref{eq:Lip} at $ \bar{d}=\mathbf{0} $ and $ \|\cdot\| \leq \|\cdot\|_{1} $, and the last inequality follows from \eqref{eq:br-threshold-2} and $ d_{u_{L}} \neq \mathbf{0} $. However, it contradicts $ \Theta^{\prime}(z;d)\geq 0$ for all $  d \in \R^{\bar{N}} $. Hence $ u_{L} = \psi_{L-1} ( \theta, u_{1}, \dots, u_{L-1} ) $.
	
	For any $ \ell=L-1,\dots,1 $, we next show $ u_{\ell} = \psi_{\ell-1} ( \theta, u_{1}, \dots, u_{\ell-1} ) $ using $ I_{+}^{\ell}(z) $, $ I_{-}^{\ell}(z) $ and $I_{0}^{\ell}(z)$ in Proposition \ref{dd_F1}. Suppose $ u_{\ell} \neq \psi_{\ell-1} ( \theta, u_{1}, \dots, u_{\ell-1} ) $, then $ I_{+}^{\ell}(z) \cup I_{-}^{\ell}(z) \neq \emptyset $. Let $ \bar{z}:=( \bar{\theta}^{\top},  \bar{u}_{1}^{\top}, \dots, \bar{u}_{L}^{\top} )^{\top} $ with
	\begin{align*}
		&\bar{\theta}:=\theta, ~
		\bar{u}_{1} := u_{1}, ~ \dots, ~ \bar{u}_{\ell-1} := u_{\ell-1}, \\
		& \bar{u}_{\ell} = \psi_{\ell-1} ( \theta, u_{1}, \dots, u_{\ell-1} ) \neq u_{\ell}, \\
		& \bar{u}_{\ell+1}:= u_{\ell+1} + \psi_{\ell}^{\prime} ( \theta, u_{1}, \dots, u_{\ell}; \mathbf{0}, \dots, \mathbf{0}, \bar{u}_{\ell} - u_{\ell}  ), \\
		& \vdots \\
		& \bar{u}_{L}:= u_{L} + \psi_{L-1}^{\prime} ( \theta, u_{1}, \dots, u_{L-1}; \mathbf{0}, \dots, \mathbf{0}, \bar{u}_{\ell} - u_{\ell} , \dots , \bar{u}_{L-1} - u_{L-1}  ),
	\end{align*}
	and $ d:= \bar{z} -z= ( \mathbf{0}, \dots, \mathbf{0}, d_{ u_{\ell} }, \dots, d_{u_{L}} ) \neq \mathbf{0} $. Then it can be checked that
	\begin{equation}
		\label{eq:4-0}
		\begin{aligned}
			d_{u_{\ell+1}} = \psi_{\ell}^{\prime} ( \theta,  \mathbf{u}_{\ell}; d_{\theta} ,   d_{ \mathbf{u}_{\ell} }  ), ~
			\dots, ~
			d_{u_{L}} = \psi_{L-1}^{\prime} ( \theta, \mathbf{u}_{L-1}; d_{\theta} ,  d_{ \mathbf{u}_{L-1} }  ).
		\end{aligned}
	\end{equation}
	Together with Proposition \ref{dd_F1}, it implies that
	\begin{align}
		\Theta^{\prime}( z;d )
		&= F^{\prime}(u;d_{u} )  + \beta_{\ell}  \left(
		\sum_{i \in I_{+}^{\ell}(z) } [ d_{u_{\ell}} ]_{i}  -  \sum_{i \in I_{-}^{\ell}(z) } [ d_{u_{\ell}} ]_{i} + \sum_{i \in I_{0}^{\ell}(z) } | [ d_{u_{\ell}} ]_{i} | \right) \notag \\
		&= g^{\prime}(u;d_{u} )  - \beta_{\ell} \| d_{u_{\ell}} \|_{1} \notag \\
		&\leq K_{g} \|d_{u} \| - \beta_{\ell} \| d_{u_{\ell}} \|_{1}, \label{eq:4-1}
	\end{align}
	where the two equalities use the definitions of $d$ and $ I_{+}^{\ell}(z), I_{-}^{\ell}(z), I_{0}^{\ell}(z)  $, the inequality holds due to \eqref{eq:Lip} at $ \bar{d}=\mathbf{0} $. Next we give an upper bound of $ \| d_{u} \| $ by estimating $ \{ d_{u_{j}} , j \in [L] \}$. Since $ d_{u_{1}}= \mathbf{0}, \dots, d_{u_{\ell-1}}= \mathbf{0} $, we only need to analyze $ \{ d_{u_{j}} , j = \ell,\dots,L \}$.
	For $j=\ell$, it follows from the definitions of $ d_{u_{\ell}} $ and $ \bar{u}_{\ell} $ that
	\begin{align}
		\label{eq:4-2}
		\| d_{u_{\ell}} \| =  \| u_{\ell} - \psi_{\ell-1} ( \theta, u_{1}, \dots,u_{\ell-1} ) \| \neq  0.
	\end{align}
	For $ j=\ell+1 $, it follows from \eqref{eq:4-0} that
	\begin{align}
		\label{eq:4-3}
		\| d_{u_{\ell+1}} \|
		&= \| \psi_{\ell}^{\prime} ( \theta, \mathbf{u}_{\ell}; d_{\theta} ,   d_{ \mathbf{u}_{\ell} }  ) \|
		\leq K_{\ell} (   \| d_{u_{1}}\| +  \dots +  \| d_{ u_{\ell} } \| )
		= K_{\ell} \| d_{ u_{\ell} } \|,
	\end{align}
	where the inequality holds due to \eqref{eq:Lip} at $ \bar{d}=\mathbf{0}, d_{\theta}=\mathbf{0} $,
	the last equality uses $ d_{u_{1}}= \mathbf{0}, \dots, d_{u_{\ell-1}}= \mathbf{0} $. And for $ j\geq \ell+2 $, assume that
	\begin{align}\label{eq:4-4}
		\|d_{k}\| \leq \left( K_{k-1}\prod_{i=\ell+1}^{k-1} (1+ K_{i-1} ) \right) \|d_{u_{\ell}} \|
	\end{align}
	for all $ k=\ell+1, \dots, j-1 $.
	Then we can obtain that \eqref{eq:4-4} also holds at $ k=j $:
	\begin{align}
		\|d_{u_{j}}\|
		&= \| \psi_{j-1}^{\prime} ( \theta, \mathbf{u}_{j-1}; d_{\theta} ,   d_{ \mathbf{u}_{j-1} }  ) \| \notag  \\
		&\leq   K_{j-1} (   \| d_{u_{1}}\| +  \dots +  \| d_{ u_{j-1} } \| ) \label{eq:4-4+} \\
		&=  K_{j-1} (   \| d_{u_{\ell}}\|+ \| d_{u_{\ell+1}}\| +  \dots +  \| d_{ u_{j-1} } \| ) \notag \\
		& \leq \left( K_{j-1}\prod_{i=\ell+1}^{j-1} (1+ K_{i-1} ) \right ) \|d_{u_{\ell}} \| ,\notag
	\end{align}
	where the first inequality holds due to \eqref{eq:Lip} at $ \bar{d}=\mathbf{0}, d_{\theta}=\mathbf{0} $,
	the second equality uses $  d_{u_{1}}= \mathbf{0}, \dots, d_{u_{\ell-1}}= \mathbf{0} $, and the second inequality uses \eqref{eq:4-4} at $ k=\ell+1, \dots, j-1 $. By induction and \eqref{eq:4-3}, it implies that \eqref{eq:4-4} holds for all $ k=\ell+1,\dots,L $. Plugging these upper bounds for $ \{ \| d_{u_{j}} \| , j \in [L] \}$ into \eqref{eq:4-1}, we have
	\begin{align*}
		\Theta^{\prime}( z;d )
		&\leq K_{g} ( \|d_{u_{\ell}} \| + \|d_{u_{\ell+1}} \| + \dots + \| d_{u_{L}} \| ) -\beta_{\ell} \|d_{u_{\ell}} \| \\
		&\leq \left( K_{g} \left( 1 + K_{\ell}
		+ \dots +  K_{L-1}\prod_{i=\ell+1}^{L-1} (1+ K_{i-1} )  \right) - \beta_{\ell}  \right) \|d_{u_{\ell}} \| \\
		&\leq  \left( K_{g} \prod_{i=\ell+1}^{L} (1+ K_{i-1} )    - \beta_{\ell}  \right) \|d_{u_{\ell}} \| \\
		& <0,
	\end{align*}
	where the last inequality uses \eqref{eq:br-threshold-2} and \eqref{eq:4-2}. However, it contradicts $ \Theta^{\prime}(z;d) \geq 0$ for all $ d \in \R^{\bar{N}} $. Hence, $ u_{\ell}=\psi_{\ell-1} ( \theta,u_{1}, \dots,u_{\ell-1} ) $, which yields the result due to the arbitrariness of $\ell$. 
\end{proof}

In general, condition \eqref{eq:br-threshold-2} can be satisfied under $\beta_{1}= \dots  =\beta_{L}=\beta $ with sufficiently large $\beta>0$, since $K_g$ and $ \{ K_{\ell}, \ell \in [L-1] \} $ defined in \eqref{def:Lipg} and \eqref{def:Lips} are non-increasing when $\beta$ is increasing.
For certain applications in machine learning, such as training process of RNNs to be shown in Section \ref{sec4}, Lipschitz moduli $K_{g}$ and $ \{ K_{\ell}, \ell\in [L-1] \} $ satisfying \eqref{eq:Lip} on $ lev_{\leq \bar{\gamma} } \Theta $ are easy to estimate (see \eqref{eq:5-1}-\eqref{eq:5-5}), which provides computable thresholds for $ \{ \beta_{\ell} , \ell \in [L] \} $.
Based on Lemma \ref{lem:D1-F0}, we could show the equivalence of \eqref{eq} and \eqref{l1pen} in terms of global optimality and d-stationarity.
\begin{theorem}\label{th:P0=P1}
	Under Assumption \ref{as1}, set $ \{ \beta_{\ell} >0, \ell \in [L] \} $ satisfying \eqref{eq:br-threshold-2}. Then
	\begin{itemize}
		\item[(a)] $ \mathcal{S}_{0} = \mathcal{S}_{1} $;
		\item[(b)] for any $z \in lev_{\leq \bar{\gamma}} \Theta $ with $\bar{\gamma}$ defined in \eqref{def:gammabar}, $z$ is a d-stationary point of \eqref{eq} if and only if it is a d-stationary point of \eqref{l1pen}.
	\end{itemize}
\end{theorem}
\begin{proof}
	$(a)$.  For any $ z \in \mathcal{S}_{1} $, it follows from Lemma \ref{lem:1op} that $ \Theta^{\prime}(z;d) \geq 0 $ for all $ d \in \R^{\bar{N}} $, i.e., $z$ is a d-stationary point of \eqref{l1pen}. Together with Lemma \ref{lem:D1-F0} and $ \Theta(z) \leq \Theta( z^{0} ) = \bar{\gamma} $ from the global optimality of $z$, it implies that $ z \in \mathcal{F}_{0} $. Hence, $  \mathcal{S}_{1} \subseteq \mathcal{F}_{0} $. Since $  \mathcal{S}_{1} \neq    \emptyset  $, we further have
	\begin{align*}
		\mathcal{S}_{1}
		= \argmin_{ z \in \R^{\bar{N}} } \Theta(z)
		= \argmin_{ z \in \mathcal{F}_{0} } \Theta(z)
		= \argmin_{ z \in \mathcal{F}_{0} } F(z)
		= \mathcal{S}_{0}.
	\end{align*}
	
	$ (b) $. We first show that any d-stationary point of \eqref{l1pen} in $ lev_{\leq \bar{\gamma}} \Theta $ must be a d-stationary point for \eqref{eq}. Firstly, it follows from Lemma \ref{lem:D1-F0} that $ z \in \mathcal{F}_{0} $. Hence, it follows from Proposition \ref{dd_F1} and Theorem \ref{th:T-express-2} that for any $ d\in \mathcal{T}_{\mathcal{F}_{0}} ( z ) $,
	\begin{align*}
		F^{\prime}(z;d) = \Theta^{\prime} ( z;d ) \geq 0,
	\end{align*}
	where the inequality comes from $ \Theta^{\prime}(z;d) \geq 0$ for all $  d $. 
	
	Next we will show the reverse implication: any d-stationary point $z$  of \eqref{eq} with $ \Theta(z) \leq  \bar{\gamma} $ is also a d-stationary point of \eqref{l1pen}. Firstly, it follows from $ z \in \mathcal{F}_{0} $ that $ I_{+}^{\ell}(z) = I_{-}^{\ell}(z)=\emptyset $ for all $
	\ell \in [L] $, which are defined in Proposition \ref{dd_F1}. It simplifies \eqref{dd1} as
	\begin{equation}
		\begin{aligned}
			\Theta^{\prime}(z;d)
			= \, & F^{\prime}(z;d) + \sum_{\ell=1}^{L} \beta_{\ell} \| d_{u_{\ell}} - \psi_{\ell-1}^{\prime}( \theta, \mathbf{u}_{\ell-1} ; d_{\theta}, d_{\mathbf{u}_{\ell-1}} ) \|_{1}  .
		\end{aligned}
		\label{eq:D0=D1-1}
	\end{equation}
	Together with Theorem \ref{th:T-express-2}, it further implies that for any $ d \in \mathcal{T}_{\mathcal{F}_{0} }(z) $,
	\begin{align}
		\Theta^{\prime}(z;d)
		=   F^{\prime}(z;d)
		\geq  0, \label{eq:D0=D1-2}
	\end{align}
	where the inequality holds since $z$ is the d-stationary point of \eqref{eq}. 
	For any $ d \notin \mathcal{T}_{\mathcal{F}_{0} }(z) $, we can construct a direction $\bar{d}$ as follows: set $ \bar{d}_{\theta}:= d_{\theta} $ and
	$ \bar{d}_{u_{\ell}}:= \psi_{\ell-1}^{\prime} ( \theta, u_{1},   \dots,  u_{\ell-1} ; \bar{d}_{\theta}, \bar{d}_{u_{1}}, \dots , \bar{d}_{u_{\ell-1}} ) $
	in the order of $ \ell=1,\dots, L $. Then by Theorem \ref{th:T-express-2}, we have $ \bar{d} \in \mathcal{T}_{\mathcal{F}_{0} }(z) $. Hence, it follows from \eqref{eq:D0=D1-1} and \eqref{eq:D0=D1-2} that
	\begin{align}
		\Theta^{\prime}(z;d)
		= \, & \Theta^{\prime}(z;d) - \Theta^{\prime}(z;\bar{d}) + \Theta^{\prime}(z;\bar{d}) \notag \\
		\geq \, & \Theta^{\prime}(z;d) - \Theta^{\prime}(z;\bar{d}) \notag \\
		= \, & F^{\prime}(z;d) - F^{\prime}(z;\bar{d})
		+  \sum_{\ell=1}^{L} \beta_{\ell} \| d_{u_{\ell}} - \psi_{\ell-1}^{\prime}( \theta, \mathbf{u}_{\ell-1} ; d_{\theta},   d_{\mathbf{u}_{\ell-1}} ) \|_{1}. \label{eq:D0=D1-3}
	\end{align}
	Next we will show that the right-hand side of \eqref{eq:D0=D1-3} is nonnegative.
	Note that
	\begin{align}\label{eq:4-4.5}
		F^{\prime}(z;d) - F^{\prime}(z;\bar{d})
		= g^{\prime}(u;d_{u}) - g^{\prime}(u;\bar{d}_{u})
		\geq -K_{g} \sum_{\ell=1}^{L} \| d_{u_{\ell}} - \bar{d}_{u_{\ell}} \|,
	\end{align}
	where the equality follows from $ \bar{d}_{\theta} = d_{\theta} $ and \eqref{dd0}, and the inequality uses \eqref{eq:Lip}. We only need to estimate $ \{  \| d_{u_{\ell}} - \bar{d}_{u_{\ell}} \| , \ell \in [L] \} $ by induction in the following.
	For $\ell=1$, it follows from the definition of $ \bar{d}_{u_{1}} $ and $ \bar{d}_{\theta} = d_{\theta} $ that
	\begin{align}\label{eq:4-5}
		\| d_{u_{1}} - \bar{d}_{u_{1}} \| = \| d_{u_{1}} - \psi_{0}^{\prime} ( \theta;d_{\theta} ) \|.
	\end{align}
	For $ \ell=2,\dots,L $, assume that
	\begin{align}
		\label{eq:4-6}
		&\|d_{u_{j}}- \bar{d}_{u_{j}} \| \\
		\leq \, & \|d_{u_{j}} -  \psi_{j-1}^{\prime}( \theta, \mathbf{u}_{j-1} ; d_{\theta} ,  d_{ \mathbf{u}_{j-1} } ) \| \notag \\
		& + K_{j-1} \sum_{k=1}^{j-1}  \left[ \prod_{i=k+1}^{j-1} ( 1+K_{i-1} )  \right]
		\| d_{u_{k}} - \psi_{k-1}^{\prime} ( \theta, \mathbf{u}_{k-1}; d_{\theta},   d_{\mathbf{u}_{k-1}} ) \| \notag
	\end{align}
	holds for all $ j=1,\dots,\ell-1 $. Then we can deduce that \eqref{eq:4-6} also holds at $j=\ell$:
	\begin{align}
		&\|d_{u_{\ell}}- \bar{d}_{u_{\ell}} \| \notag \\
		= \, &  \|d_{u_{\ell}} -  \psi_{\ell-1}^{\prime}( \theta, \mathbf{u}_{\ell-1} ; \bar{d}_{\theta} ,   \bar{d}_{ \mathbf{u}_{\ell-1} } ) \| \notag\\
		\leq \, &  \|d_{u_{\ell}} -  \psi_{\ell-1}^{\prime}( \theta, \mathbf{u}_{\ell-1} ; d_{\theta} ,   d_{ \mathbf{u}_{\ell-1} } ) \|  \notag\\
		&+ K_{\ell-1} (   \| d_{u_{1}} - \bar{d}_{u_{1}} \| + \dots  + \| d_{u_{\ell-1}} - \bar{d}_{u_{\ell-1}} \| ) \label{eq:4-6.5} \\
		\leq \, & \|d_{u_{\ell}} -  \psi_{\ell-1}^{\prime}( \theta, \mathbf{u}_{\ell-1} ; d_{\theta} ,   d_{ \mathbf{u}_{\ell-1} } ) \| \notag \\
		&+ K_{\ell-1}
		\begin{pmatrix}
			\|d_{u_{1}}- \psi_{0}^{\prime}( \theta;d_{\theta} ) \| \cdot [  1 + K_{1} + \cdots+ K_{\ell-2} \prod_{i=2}^{\ell-2} ( 1+K_{i-1} ) ]\\
			+\dots + \|d_{u_{\ell-1}} -  \psi_{\ell-2}^{\prime}( \theta, \mathbf{u}_{\ell-2} ; d_{\theta} ,  d_{ \mathbf{u}_{\ell-2} } ) \|
		\end{pmatrix}\notag\\
		= \, & \|d_{u_{\ell}} -  \psi_{\ell-1}^{\prime}( \theta, \mathbf{u}_{\ell-1} ; d_{\theta} ,  d_{ \mathbf{u}_{\ell-1} } ) \|  \notag\\
		&+ K_{\ell-1} \cdot
		\begin{pmatrix}
			\|d_{u_{1}}- \psi_{0}^{\prime}( \theta;d_{\theta} ) \| \cdot \prod_{i=2}^{\ell-1} ( 1+K_{i-1} ) +\dots\\
			+ \|d_{u_{\ell-1}} -  \psi_{\ell-2}^{\prime}( \theta, \mathbf{u}_{\ell-2} ; d_{\theta} ,   d_{ \mathbf{u}_{\ell-2} } ) \|
		\end{pmatrix} , \notag
	\end{align}
	where the first equality comes from the definition of $ \bar{d}_{u_{\ell}} $, the first inequality uses $ d_{\theta} = \bar{d}_{\theta} $ and \eqref{eq:Lip}, the last inequality follows from \eqref{eq:4-6} at $ j=1,\dots,\ell-1 $.
	Together with \eqref{eq:4-5}, it implies that \eqref{eq:4-6} holds for all $ \ell\in [L] $. Plugging these upper bounds for $ \{ \|d_{u_{\ell}} - \bar{d}_{ u_{\ell} }  \|, \ell \in [L]  \} $ into \eqref{eq:4-4.5}, we have
	\begin{align*}
		&F^{\prime}(z;d)  -  F^{\prime}(z;\bar{d}) \\
		\geq   & - K_{g}
		\begin{pmatrix}
			\|d_{u_{1}}- \psi_{0}^{\prime}( \theta;d_{\theta} ) \| \cdot [  1 + K_{1} + \cdots+ K_{L-1 } \prod_{i=2}^{L-1} ( 1+K_{i-1} ) ]\\
			+\dots + \|d_{u_{L }} -  \psi_{L-1 }^{\prime}( \theta, \mathbf{u}_{L-1} ; d_{\theta} ,  d_{ \mathbf{u}_{L-1} } ) \|
		\end{pmatrix} \\
		= & - K_{g} \sum_{\ell=1}^{L} \left[ \prod_{j=\ell+1}^{L} (1+ K_{j-1}) \right]
		\| d_{u_{\ell }} -  \psi_{\ell-1 }^{\prime}( \theta, \mathbf{u}_{\ell-1} ; d_{\theta} ,  d_{ \mathbf{u}_{\ell-1} } )\|.
	\end{align*}
	Together with \eqref{eq:D0=D1-3}, it implies that
	\begin{align}
		\Theta^{\prime}(z;d)
		\geq \, &\sum_{\ell=1}^{L} \left( \beta_{\ell} - K_{g}  \left[ \prod_{j=\ell+1}^{L} (1+ K_{j-1}) \right] \right)
		\cdot  \| d_{u_{\ell }} -  \psi_{\ell-1 }^{\prime}( \theta, \mathbf{u}_{\ell-1} ; d_{\theta} ,  d_{ \mathbf{u}_{\ell-1} } )\| \notag \\
		> \, & 0, \label{eq:4-7}
	\end{align}
	where the last inequality uses \eqref{eq:br-threshold-2}, $d \notin \mathcal{T}_{\mathcal{F}_{0}}(z) $ and Theorem \ref{th:T-express-2}. Therefore, it yields that $ \Theta^{\prime}(z;d) \geq 0 $ for all $d$, meaning that $z$ is a d-stationary point of \eqref{l1pen}. 
\end{proof}

\begin{remark}\label{remark4.1}
	Theorem \ref{th:P0=P1} is different from Theorem 2.1 of \cite{DNN-CHP} in two aspects.
	\begin{itemize}
		\item[(i)] Theorem 2.1 of \cite{DNN-CHP} only obtains one-sided implication that a d-stationary point of the penalty problem must be a d-stationary point of the original problem, while we have the equivalence. As a consequence, the penalization preserves all the d-stationary points of \eqref{eq:1.1} and \eqref{eq}. And when $ g, \{\psi_{\ell-1},\ell \in [L] \} $ are smooth or have DC structures of the pointwise max type \cite[Condition C2]{nMM-CPS}, d-stationary points of \eqref{l1pen} can be obtained by trust region methods \cite{Y141} or majorization minimization frameworks \cite{nMM-CPS}.
		
		\item[(ii)] Motivated by \cite[Theorem 2.5]{LRP-LLC} and \cite[Lemma 9]{LReLU-LLC}, we replace the boundedness requirement of $z$ with $ \Theta(z) \leq \Theta( z^{0}) $, which is easier to check since $\Theta( z^{0})$ is easy to calculate. And the condition $ \Theta(z) \leq \Theta( z^{0}) $ also helps to obtain the threshold-like conditions expressed by $ \{K_{\ell}, \ell\in[L-1]  \} $ and  $K_{g} $, which provides the relations between penalty thresholds and the number of layers $ L $.
	\end{itemize}
\end{remark}

\subsection{Second-order d-stationarity} \label{sec3.3}
Here we compare the sets of second-order d-stationary points $ \mathcal{SD}_{0} $ and $ \mathcal{SD}_{1} $ of \eqref{eq} and \eqref{l1pen} to exhibit their differences in depicting second-order necessary conditions for \eqref{eq:1.1}.

First of all, the idea of using second-order conditions of reformulated problems to characterize optimality conditions of the original problem is motivated by \cite[Proposition 9.4.2]{CP-book} for \eqref{eq:1.1} and \eqref{l1pen} with $L=1$, i.e., $ \min_{\theta} h(G(\theta)) $ with $ G(\theta):= ( \psi_{0}(\theta)^{\top} , \lambda \|\theta \|^{2} )^{\top}$ and $ h(y):= g( [y]_{1:N_{1}} ) + [y]_{(N_{1} + 1) } $.
The second-order necessary conditions in (9.41) of \cite[Proposition 9.4.2]{CP-book} can be reorganized as: if $ \theta $ is a local minimizer of \eqref{eq:1.1}, then for all $ \rho>\max\{ \overline{K}_{g}, 1 \} $ where $ \overline{K}_{g}>0 $ is a Lipschitz constant of $ g $ near $ u_{1} := \psi_{0}( \theta ) $,
\begin{align*}
	&g^{\prime}(u_{1}; d_{u_{1}} ) + 2 \lambda \theta^{\top} d_{\theta} 
	\geq 0
	\text{ for all } d=(d_{\theta}^{\top}, d_{u_{1}}^{\top})^{\top} \text{ with }  d_{u_{1}} = \psi_{0}^{\prime} ( \theta ; d_{\theta} ),  \\
	& \text{and } g^{(2)}(u_{1} ; d_{u_{1}} ) + 2 \lambda\rho \|d_{\theta} \|^{2}  + \rho  \|  \psi_{0}^{(2)} ( \theta ; d_{\theta} )   \|_{1} \geq 0 , \, \\
	& ~~~~~\text{ for all } d=(d_{\theta}^{\top}, d_{u_{1}}^{\top})^{\top} \text{ with }
	d_{u_{1}} = \psi_{0}^{\prime} ( \theta ; d_{\theta} ),  \,
	g^{\prime}(u_{1}; d_{u_{1}} ) + 2 \lambda \theta^{\top} d_{\theta} =0,
\end{align*}
which is actually covered by our second-order necessary conditions in constructing $ \mathcal{SD}_{1} $ 
since for any $ z=( \theta^{\top} , u_{1}^{\top} )^{\top} $ with $u_{1} = \psi_{0}( \theta )$, 
\begin{itemize}
	\item for any $ d=(d_{\theta}^{\top}, d_{u_{1}}^{\top})^{\top} $ with $ d_{u_{1}} = \psi_{0}^{\prime} ( \theta ; d_{\theta} ) $, it follows from Proposition \ref{dd_F1} that $ \Theta^{\prime}(z;d)  = g^{\prime}(u_{1}; d_{u_{1}} ) + 2 \lambda \theta^{\top} d_{\theta} $;
	\item it follows from Proposition \ref{dd_F1} that
	\begin{align*}
		\{ d \mid d_{u_{1}} = \psi_{0}^{\prime} ( \theta ; d_{\theta} ),
		g^{\prime}(u_{1}; d_{u_{1}} ) + 2 \lambda \theta^{\top} d_{\theta} =0 \}
		\subseteq
		\{ d \mid \Theta^{\prime}(z;d) =0  \};
	\end{align*}
	\item for any $ d  $ with $ d_{u_{1}} = \psi_{0}^{\prime} ( \theta ; d_{\theta} )$ and $ g^{\prime}(u_{1}; d_{u_{1}} ) + 2 \lambda \theta^{\top} d_{\theta} =0 $, it follows from Proposition \ref{2dd_F1} with the setting of $ \beta_{1}=\rho $ and $ \rho>\max \{ \overline{K}_{g},1 \} $ that
	\begin{align*}
		\Theta^{(2)}( z; d )&= g^{(2)}(u_{1} ; d_{u_{1}} ) + 2 \lambda \|d_{\theta} \|^{2}  + \rho  \|  \psi_{0}^{(2)} ( \theta ; d_{\theta} )   \|_{1}  \\
		&< g^{(2)}(u_{1} ; d_{u_{1}} ) + 2 \lambda\rho \|d_{\theta} \|^{2}  + \rho  \|  \psi_{0}^{(2)} ( \theta ; d_{\theta} )   \|_{1}.
	\end{align*}
\end{itemize}
Hence, we will focus on our second-order necessary conditions for \eqref{eq} and \eqref{l1pen} rather than generalizing (9.41) of \cite[Proposition 9.4.2]{CP-book} to the case of \eqref{eq:1.1}.

By the results in Sections \ref{sec3.1} and \ref{sec3.2}, the second-order necessary conditions of \eqref{eq} and \eqref{l1pen} specified in Definition \ref{def:1op} are both able to characterize solutions of \eqref{eq:1.1}.
It follows from Lemma \ref{lem:1op} and Theorem \ref{th:P0=P1} that
\begin{align*}
	\begin{matrix}
		\mathcal{S}_{0} & \subseteq & \mathcal{SD}_{0} \cap lev_{\leq \bar{\gamma}} \Theta & \subseteq&  \mathcal{D}_{0} \cap lev_{\leq \bar{\gamma}} \Theta  \\
		|| & & & & ||  \\
		\mathcal{S}_{1} & \subseteq & \mathcal{SD}_{1} \cap lev_{\leq \bar{\gamma}} \Theta & \subseteq&  \mathcal{D}_{1} \cap lev_{\leq \bar{\gamma}} \Theta.
	\end{matrix}
\end{align*}
Together with the bijection between $ \mathcal{S} $ and $ \mathcal{S}_{0} $ (discussed after Lemma \ref{lem:bd-opt-0}), it implies that for any $ \theta \in \mathcal{S} $, the point $ z=(\theta^{\top}, u_{1}^{\top}, \dots , u_{L}^{\top}  )^{\top} $ must belong to $ \mathcal{SD}_{0} \cap lev_{\leq \bar{\gamma}} \Theta $ and $  \mathcal{SD}_{1} \cap lev_{\leq \bar{\gamma}} \Theta $ where $ u_{\ell} := \psi_{\ell-1} ( \theta , u_{1}, \dots, u_{\ell-1} ) $ for all $ \ell \in [L] $. 
Furthermore, we could find the latter condition $z \in \mathcal{SD}_{1}  \cap lev_{\leq \bar{\gamma}} \Theta$ is stronger by the following observation.
\begin{theorem}
	\label{th:SD0-SD1}
	Under Assumptions \ref{as1}, \ref{as2} and \eqref{eq:br-threshold-2},  $ \mathcal{SD}_{0} \cap lev_{\leq \bar{\gamma}} \Theta \supseteq \mathcal{SD}_{1} \cap lev_{\leq \bar{\gamma}} \Theta $.
\end{theorem}
\begin{proof}
	By the definitions of $ \mathcal{SD}_{0} $ and $\mathcal{SD}_{1}$,
	\begin{align*}
		&\mathcal{SD}_{0} \cap lev_{\leq \bar{\gamma}} \Theta =
		\{ z \mid F^{(2)} (z;d)\geq 0, \forall d \in P_{\mathcal{F}_{0}} (z) \text{ with } F^{\prime}(z;d) = 0 \} \cap ( \mathcal{D}_{0} \cap lev_{\leq \bar{\gamma}} \Theta ), \\
		&\mathcal{SD}_{1} \cap lev_{\leq \bar{\gamma}} \Theta =
		\{ z \mid \Theta^{(2)} (z;d)\geq 0, \forall d  \text{ with } \Theta^{\prime}(z;d) = 0 \} \cap ( \mathcal{D}_{1} \cap lev_{\leq \bar{\gamma}} \Theta ).
	\end{align*}
	Together with $ \mathcal{D}_{0} \cap lev_{\leq \bar{\gamma}} \Theta = \mathcal{D}_{1} \cap lev_{\leq \bar{\gamma}} \Theta $ from Theorem \ref{th:P0=P1}, we only need to prove that for any $z \in \mathcal{D}_{1} \cap lev_{\leq \bar{\gamma}} \Theta $ satisfying
	$
	\Theta^{(2)} (z;d)\geq 0 \text{ for all } d  \text{ with } \Theta^{\prime}(z;d) = 0,
	$
	the inequality $ F^{(2)} (z;d)\geq 0 $ holds for all $ d \in P_{\mathcal{F}_{0}} (z) $ with $ F^{\prime}(z;d) = 0 $.
	
	First we show
	\begin{align}\label{eq:3.3.1}
		\psi_{\ell-1}^{(2)}(\theta, \mathbf{u}_{\ell-1}; d_{\theta} ,   d_{\mathbf{u}_{\ell-1}} )=0 , \, \ell\in[L]
		\text{  for all } z \in \mathcal{F}_{0}, d \in P_{\mathcal{F}_{0}} (z)
	\end{align}
	by contradiction.
	If there exists $ \ell\in[L] $ and $ i \in [N_{\ell}] $, such that $ [\psi_{\ell-1}^{(2)}(\theta, \mathbf{u}_{\ell-1};   d_{\theta} ,   d_{\mathbf{u}_{\ell-1}}  ) ]_{i} >0 $, then it follows from the definition of second-order directional derivatives that for any sufficiently small positive number $ \tau $,
	\begin{align*}
		0< \, &[ \psi_{\ell-1} (\theta + \tau d_{\theta} ,   \mathbf{u}_{\ell-1}+ \tau d_{\mathbf{u}_{\ell-1}} )  ]_{i}
		-  [ \psi_{\ell-1} (\theta, \mathbf{u}_{\ell-1})  ]_{i}
		- \tau [\psi_{\ell-1}^{\prime}(\theta, \mathbf{u}_{\ell-1}; d_{\theta} ,   d_{\mathbf{u}_{\ell-1}} ) ]_{i} .
	\end{align*}
	Together with $ z \in \mathcal{F}_{0}, d \in P_{\mathcal{F}_{0}} (z) \subseteq \mathcal{T}_{ \mathcal{F}_{0} }  (z) $, it implies that
	\begin{align*}
		& [ u_{\ell} + \tau d_{u_{\ell}} ]_{i} - [ \psi_{\ell-1} (\theta + \tau d_{\theta} ,  \mathbf{u}_{\ell-1}+ \tau d_{\mathbf{u}_{\ell-1}} )  ]_{i}
		\\
		= \, &
		[ \psi_{\ell-1} (\theta, \mathbf{u}_{\ell-1})  ]_{i}
		+ \tau [\psi_{\ell-1}^{\prime}(\theta, \mathbf{u}_{\ell-1}; d_{\theta} ,   d_{\mathbf{u}_{\ell-1}} ) ]_{i}
		-[ \psi_{\ell-1} (\theta + \tau d_{\theta} ,   \mathbf{u}_{\ell-1}+ \tau d_{\mathbf{u}_{\ell-1}} )  ]_{i} \\
		< \, & 0
	\end{align*}
	for any sufficiently small positive $ \tau $, which contradicts with $  d \in P_{\mathcal{F}_{0}} (z) $. Similar contradiction appears if there exists $ \ell\in[L] $ and $ i \in [N_{\ell}] $, such that $ [\psi_{\ell-1}^{(2)}(\theta, \mathbf{u}_{\ell-1}; d_{\theta} ,    d_{\mathbf{u}_{\ell-1}} ) ]_{i}  < 0 $, which yields \eqref{eq:3.3.1}.
	
	Hence, for any $z \in \mathcal{SD}_{1} \cap lev_{\leq \bar{\gamma}} \Theta  $, we have
	\begin{align*}
		0 \leq \Theta^{(2)} (z;d)
		&= F^{(2)} (z;d) + \sum_{\ell=1}^{L} \beta_{\ell}  \| \psi_{\ell-1}^{(2)}(\theta, \mathbf{u}_{\ell-1}; d_{\theta},   d_{ \mathbf{u}_{\ell-1} }  ) \|_{1}
		= F^{(2)} (z;d)
	\end{align*}
	for all $ d \in P_{\mathcal{F}_{0}} (z)  $ with $ F^{\prime}(z;d) = 0 $, where the inequality is derived by
	\begin{align*}
		\{ d \in P_{\mathcal{F}_{0}} (z) \mid F^{\prime}(z;d) = 0 \} \subseteq
		\{ d \in \mathcal{T}_{\mathcal{F}_{0}}  (z) \mid F^{\prime}(z;d) = 0 \} \subseteq
		\{ d   \mid \Theta^{\prime}(z;d) = 0 \}
	\end{align*}
	from $z \in   \mathcal{F}_{0} $, Theorem \ref{th:T-express-2} and Proposition \ref{dd_F1}, the first equality uses $ z \in \mathcal{F}_{0} $, Theorem \ref{th:T-express-2} and Proposition \ref{2dd_F1}, and the last equality uses \eqref{eq:3.3.1}.
\end{proof} 
\begin{remark}\label{remark:SD0=SD1}
	Here we discuss the conditions under which the equality in Theorem \ref{th:SD0-SD1} holds.
	Under the premises of Theorem \ref{th:SD0-SD1}, it follows from \eqref{eq:4-7} that $ \Theta^{\prime}(z;d)>0 $ for any $ z \in \mathcal{D}_{0} \cap lev_{\leq\bar{\gamma}} \Theta = \mathcal{D}_{1} \cap lev_{\leq\bar{\gamma}} \Theta $ and any $ d \notin \mathcal{T}_{ \mathcal{F}_{0} }(z) $. Hence, for any $ z \in \mathcal{D}_{0} \cap lev_{\leq\bar{\gamma}} \Theta = \mathcal{D}_{1} \cap lev_{\leq\bar{\gamma}} \Theta $,
	\begin{align}
		\label{eq:4-8-0}
		\{d \mid  \Theta^{\prime}(z;d)=0 \}
		= \{d \in \mathcal{T}_{ \mathcal{F}_{0} }(z) \mid  \Theta^{\prime}(z;d)=0 \}
		= \{d \in \mathcal{T}_{ \mathcal{F}_{0} }(z) \mid  F^{\prime}(z;d)=0 \},
	\end{align}
	where the last equality uses $ z\in \mathcal{F}_{0} $, Theorem \ref{th:T-express-2} and Proposition \ref{dd_F1}. Together with Proposition \ref{2dd_F1} and Theorem \ref{th:T-express-2}, it implies that for any $d$ with $ \Theta^{\prime}(z;d)=0 $,
	\begin{align*}
		\Theta^{(2)}(z;d)= F^{(2)}(z;d) + \sum_{\ell=1}^{L} \beta_{\ell}  \| \psi_{\ell-1}^{(2)}(\theta, \mathbf{u}_{\ell-1}; d_{\theta},   d_{ \mathbf{u}_{\ell-1} }  ) \|_{1},
	\end{align*}
	which indicates that $ z \in \mathcal{SD}_{1} \cap lev_{\leq \bar{\gamma}} \Theta $ if and only if $  z \in \mathcal{D}_{1} \cap lev_{\leq \bar{\gamma}} \Theta $ and
	\begin{align}
		\label{eq:4-8-1}
		F^{(2)}(z;d) + \sum_{\ell=1}^{L} \beta_{\ell}  \| \psi_{\ell-1}^{(2)}(\theta, \mathbf{u}_{\ell-1}; d_{\theta},   d_{ \mathbf{u}_{\ell-1} }  ) \|_{1} \geq 0
	\end{align}
	for all $ d \in \mathcal{T}_{ \mathcal{F}_{0} }(z) $ satisfying $ F^{\prime}(z;d) = 0 $. Then it follows from \eqref{eq:3.3.1} and $ P_{\mathcal{F}_{0}}(z) \subseteq \mathcal{T}_{\mathcal{F}_{0}}(z) $ that, the equality in Theorem \ref{th:SD0-SD1} holds if and only if
	\begin{equation}
		\label{eq:4-8}
		\begin{aligned}
			& F^{(2)}(z;d) + \sum_{\ell=1}^{L} \beta_{\ell}  \| \psi_{\ell-1}^{(2)}(\theta, \mathbf{u}_{\ell-1}; d_{\theta},   d_{ \mathbf{u}_{\ell-1} }  ) \|_{1} \geq 0, \\
			& \text{for all } d \in \mathcal{T}_{ \mathcal{F}_{0} }(z) \backslash P_{\mathcal{F}_{0}}(z)
			\text{ satisfying } F^{\prime}(z;d) = 0.
		\end{aligned}
	\end{equation}
	Sufficient conditions for \eqref{eq:4-8} include the following two conditions.
	\begin{itemize}
		\item $ P_{\mathcal{F}_{0}}(z) = \mathcal{T}_{\mathcal{F}_{0}}(z) $. If $ \mathcal{F}_{0}  $ is a polyhedron or a union of finite number of polyhedrons, then $ P_{\mathcal{F}_{0}}(z) = \mathcal{T}_{\mathcal{F}_{0}}(z) $ holds for any $ z \in \mathcal{F}_{0} $. For example, under the setting of $ \psi_{0}(\theta):= a^{\top} \theta , \psi_{1}(\theta,u_{1}):=[u_{1}]_{+} $ for a vector $ a \in \R^{n} $,
		\begin{align*}
			\mathcal{F}_{0}
			&=\{ z=(\theta^{\top}, u_{1}, u_{2})^{\top} \mid u_{1} = a^{\top} \theta ,\, u_{2}=[u_{1}]_{+} \}\\
			&=\{ z \mid u_{1} = a^{\top} \theta ,\, u_{1}\geq 0,\,  u_{2}= u_{1}  \}
			\cup
			\{ z \mid u_{1} = a^{\top} \theta ,\, u_{1}\leq 0,\,  u_{2}= 0  \}
		\end{align*}
		is a union of two polyhedrons.
		\item $F$ is convex and twice directionally differentiable. Since $ F(z)=g(u)+ \lambda \|\theta \|^{2} $, $F$ is convex and twice directionally differentiable when $g$ is convex and twice directionally differentiable. In this case, \eqref{eq:4-8} naturally holds since $ F^{(2)}(z;d) = \lim_{\tau \downarrow 0 } \frac{F(z+\tau d) - F(z) - F^{\prime}(z;\tau d) }{ \tau^{2}/2 }    \geq 0 $ for all $d$.
	\end{itemize}
\end{remark}
Inspired by Remark \ref{remark:SD0=SD1}, we can provide an example where $ \mathcal{SD}_{0} \cap lev_{\leq \bar{\gamma}} \Theta \supsetneq \mathcal{SD}_{1} \cap lev_{\leq \bar{\gamma}} \Theta $ under the premises of Theorem \ref{th:SD0-SD1}.
\begin{example}\label{ex2}
	Consider \eqref{eq:1.1} with $ L=2, n=N_{1}=N_{2}=1 $ and $ \lambda=0.01 $,
	\begin{align*}
		\psi_{0}(\theta):= \theta, ~
		\psi_{1}(\theta,u_{1}):= u_{1}^{2}, ~
		g(u_{1}, u_{2}) := [-u_{1}^{2}+ 0.5 u_{2} + 0.0001]_{+}.
	\end{align*}
	On the one hand, it can be verified that $ z^{0}=(0,\psi_{0}(0),  \psi_{1}(0,\psi_{0}(0)))^{\top} = \mathbf{0} $ is a second-order d-stationary point of \eqref{eq}. Firstly, it follows from Theorem \ref{th:T-express-2}, the definition of $ P_{\mathcal{F}_{0}}(\cdot) $ and \eqref{eq:3.3.1} that $ \mathcal{T}_{\mathcal{F}_{0}}(\mathbf{0})= \{ (d_{\theta} , d_{u_{1}} , d_{u_{2}} )^{\top} \mid d_{u_{1}} = d_{\theta},\, d_{u_{2}}=0  \} $ and $ P_{\mathcal{F}_{0}}(\mathbf{0}) = \{\mathbf{0}\}  $ since $ \{\mathbf{0}\} \subseteq P_{\mathcal{F}_{0}} (\mathbf{0}) \subseteq \{ d \in \mathcal{T}_{\mathcal{F}_{0}}(\mathbf{0}) \mid 2 d_{u_{1}}^{2} =0 \} = \{\mathbf{0}\} $. Then it can be verified that $ F^{\prime}(\mathbf{0};d )= 0.5 d_{u_{2}}=0 \geq 0 $ for all $d \in \mathcal{T}_{\mathcal{F}_{0}}(\mathbf{0}) $, and
	$ F^{(2)}(\mathbf{0};d ) = -2 d_{u_{1}}^{2} + 0.02 d_{\theta}^{2} =0 \geq 0 $  for all $ d \in P_{\mathcal{F}_{0}}(\mathbf{0}) \cap \{ d \mid F^{\prime}(\mathbf{0};d )=0 \} = \{\mathbf{0}\} $.
	On the other hand, $ z^{0} $ is not a second-order d-stationary point of \eqref{l1pen} with $ \beta_{1}=1, \beta_{2}=0.6 $ where \eqref{eq:br-threshold-2} holds.
	First we can verify \eqref{eq:br-threshold-2} holds, i.e., $ \beta_{1}>K_{g} (1+ K_{1} ) $ and $ \beta_{2}> K_{g} $.
	Since $ \bar{\gamma}=F(z^{0})= 10^{-4} $, for all $ ( \theta, u_{1} , u_{2} )^{\top} \in lev_{\leq\bar{\gamma}} \Theta $, it can be calculated that $ |u_{1} | \leq |\theta | + |u_{1}- \theta | \leq \sqrt{10^{-4}/10^{-2}}+  10^{-4} =0.1001  $. It implies that for all $ ( \theta, u_{1} , u_{2} )^{\top}, ( \theta, \bar{u}_{1} , \bar{u}_{2} )^{\top} \in lev_{\leq\bar{\gamma}} \Theta $,
	\begin{align*}
		&| g(u) - g(\bar{u}) | 
		\leq \sqrt{ (u_{1} + \bar{u}_{1})^{2} + 0.25 } \| u - \bar{u} \|
		< 0.5386 \| u - \bar{u} \|, \\
		& | \psi_{1}(\theta, u_{1}) - \psi_{1}(\theta, \bar{u}_{1}) |
		\leq |u_{1} + \bar{u}_{1}|\cdot |  u_{1} - \bar{u}_{1} |
		< 0.21  |  u_{1} - \bar{u}_{1} |.
	\end{align*}
	Hence, there exist $ K_{g} \in  (0, 0.5386 ] $ and $ K_{1} \in (0, 0.21] $ satisfying \eqref{def:Lipg}-\eqref{def:Lips}, which guarantees \eqref{eq:br-threshold-2}.
	Then, it follows from Theorem \ref{th:P0=P1} and $ z^{0} \in \mathcal{D}_{0} \cap lev_{\leq \bar{\gamma}} \Theta $ that $ z^{0} \in \mathcal{D}_{1} \cap lev_{\leq \bar{\gamma}} \Theta $. Together with Remark \ref{remark:SD0=SD1}, it implies that $ z^{0} \in \mathcal{SD}_{1} $ if and only if \eqref{eq:4-8-1} holds at $z^{0}$ for all $ d \in \mathcal{T}_{\mathcal{F}_{0}}(z^{0}) $ with $ F^{\prime}(z^{0} ; d )=0 $.
	However, for any $ d=(d_{\theta}, d_{u_{1}},d_{u_{2}})^{\top}  $ with $ d_{u_{2}}=0 $ and $ d_{\theta}= d_{u_{1}} \neq 0 $, we have $ d \in \mathcal{T}_{\mathcal{F}_{0}}(z^{0}) $, $ F^{\prime}(z^{0} ; d )=0 $ and
	$ F^{(2)}(z^{0};d) + \sum_{\ell=1}^{L} \beta_{\ell}  \| \psi_{\ell-1}^{(2)}(\theta^{0}, \mathbf{u}_{\ell-1}^{0}; d_{\theta},   d_{ \mathbf{u}_{\ell-1} }  ) \|_{1}
	=    -2 d_{u_{1}}^{2} + 0.02 d_{\theta}^{2} + 1.2 d_{u_{1}}^{2}
	= - 0.78 d_{\theta}^{2} <0 $, which violates \eqref{eq:4-8-1}.
\end{example}

In Section \ref{sec4}, we will provide an application of \eqref{eq:1.1} where the second-order d-stationary points of corresponding \eqref{eq} and \eqref{l1pen} are computable by certain algorithms.

\subsection{Second-order sufficient condition}\label{sec3.4}
Inspired by \cite[Proposition 9.4.2 (b)]{CP-book}, we provide second-order sufficient conditions for strong local minimizers \cite[Section 6.4]{CP-book} of \eqref{eq:1.1} in this subsection.
For a function $ f: \mathcal{F} \subseteq \R^{m} \rightarrow \R$, we say $ x \in \mathcal{F} $ is a strong local minimizer of $f$ on $ \mathcal{F} $ if there exist $ \epsilon_{1},\epsilon_{2}>0 $ such that $ f(\bar{x}) \geq f(x) + \epsilon_{1} \| \bar{x} - x \|^{2} $ for all $ \bar{x} \in \mathcal{F} $ satisfying $ \| \bar{x} - x \| \leq \epsilon_{2} $.
To this end, we need the following assumption about twice semidifferentiability \cite[Definition 13.6]{rockafellar2009variational}.
\begin{assumption}\label{as3}
	Function $g $ and each component of vector functions $\{ \psi_{\ell-1} , \ell \in [L] \} $ are twice semidifferentiable on $ \R^{\bar{N}_L} $ and $ \R^{ n +  \bar{ N}_{\ell-1} }, \ell \in [L] $ respectively.
\end{assumption}
Assumption \ref{as3} is stronger than Assumption \ref{as2}.
As shown in Lemma \ref{lem:1op}, Assumption \ref{as2} provides an upper bound of $ \liminf_{  \tau \downarrow 0 , d^{\prime} \rightarrow d    } \frac{ \Theta(z+\tau d^{\prime} ) - \Theta (z) }{ \tau^{2} /2 } $ along certain directions, whereas \eqref{eq:4-11}-\eqref{eq:4-13} indicate that Assumption \ref{as3} can simultaneously offer a lower bound for it in all directions.
For any twice semidifferentiable function $f$, we have $\text{d} f(x)(d) := \lim_{ \tau \downarrow 0 ,  d^{\prime} \rightarrow d } \frac{ f(x + \tau d^{\prime} ) - f(x  ) }{ \tau }
= f^{\prime} ( x;d )$ for any $ x,d $, and
\begin{equation}
	\label{eq:4-9}
	\begin{aligned}
		\lim_{ \substack{\tau \downarrow 0 \\  d^{\prime} \rightarrow d } }\frac{ f(x + \tau d^{\prime} ) - f(x  ) - \tau \text{d} f (x) (d^{\prime} ) }{ \tau^{2} /2 }
		& = \lim_{ \substack{\tau \downarrow 0 \\  d^{\prime} \rightarrow d } }\frac{ f(x + \tau d^{\prime}  ) - f(x  ) - \tau  f^{\prime} (x ; d^{\prime}  ) }{ \tau^{2} /2 }  \\
		& = \lim_{ \tau \downarrow 0  }\frac{ f(x + \tau d  ) - f(x  ) - \tau f^{\prime} ( x;d ) }{ \tau^{2} /2 } \\
		&= f^{(2)} ( x;d ) ,
	\end{aligned}
\end{equation}
for any $ x,d $.
For any twice semidifferentiable functions $ f_{1},f_{2} $ on $\R^{n}$ and any $ a_{1}, a_{2} \in \R $, the combination $ a_{1} f_{1} + a_{2} f_{2} $ is twice semidifferentiable  on $\R^{n}$.
Then based on previous subsections, we have the following second-order sufficient conditions for \eqref{eq:1.1}.
\begin{theorem}
	\label{th:suf-con}
	Under Assumptions \ref{as1}, \ref{as3} and \eqref{eq:br-threshold-2}, for any $ z=(\theta^{\top} , u^{\top} )^{\top} \in lev_{\leq \bar{\gamma} } \Theta $, if
	\begin{equation}
		\label{eq:suf-con}
		\begin{aligned}
			& \Theta^{\prime}(z;d) \geq 0  \text{ for all $d$} , \\
			& \text{and } F^{(2)} (z;d) - \sum_{\ell=1}^{L} \beta_{\ell}  \| \psi_{\ell-1}^{(2)} ( \theta, \mathbf{u}_{\ell-1} ; d_{\theta} , d_{\mathbf{u}_{\ell-1}} ) \|_{1} > 0 , \\
			&~~~~\text{ for all } d \neq \mathbf{0} \text{ with } \Theta^{\prime}(z;d) = 0,
		\end{aligned}
	\end{equation}
	then $z$ is a strong local minimizer of \eqref{l1pen} and $\theta$ is a strong local minimizer of \eqref{eq:1.1}.
\end{theorem}
\begin{proof}
	We first prove that $z$ is a strong local minimizer of \eqref{l1pen}. 
	According to \cite[Theorem 13.24 (c)]{rockafellar2009variational}, it is equivalent to prove that $ \mathbf{0} \in \partial \Theta (z) $ and
	\begin{align}
		\label{eq:4-10}
		\liminf_{\substack{ \tau \downarrow 0 \\ d^{\prime} \rightarrow d  } } \frac{ \Theta(z+\tau d^{\prime} ) - \Theta (z) }{ \tau^{2} /2 } >0
		~\text{ for all } d \neq \mathbf{0}.
	\end{align}
	Since $ \mathbf{0} \in \partial \Theta (z) $ can be obtained by $\hat{\partial} \Theta (z) \subseteq \partial \Theta (z) $ from \cite[Theorem 8.6]{rockafellar2009variational} and $ \mathbf{0} \in \hat{\partial} \Theta (z) $ from $ \Theta^{\prime}(z;d)\geq 0 $ for all $d$ and \cite[Exercise 8.4]{rockafellar2009variational}, we only need to prove \eqref{eq:4-10}. 
	For any $ d \neq \mathbf{0} $ satisfying $ \Theta^{\prime}(z;d) > 0 $, it follows from Assumption \ref{as1} that $ \lim_{ \tau \downarrow 0, d^{\prime} \rightarrow d   }  [ \Theta(z+\tau d^{\prime} ) - \Theta (z) ] / \tau   $ exists and equals to $\Theta^{\prime}(z;d)$, which implies that
	\begin{align*}
		\liminf_{\substack{ \tau \downarrow 0 \\ d^{\prime} \rightarrow d  } } \frac{ \Theta(z+\tau d^{\prime} ) - \Theta (z) }{ \tau^{2} /2 }
		= \liminf_{\substack{ \tau \downarrow 0 \\ d^{\prime} \rightarrow d  } } \frac{ [ \Theta(z+\tau d^{\prime} ) - \Theta (z) ] / \tau }{ \tau  /2 }
		= + \infty >0.
	\end{align*}
	For any $ d \neq \mathbf{0} $ satisfying $ \Theta^{\prime}(z;d) = 0 $, it follows from $ \Theta^{\prime}(z;d^{\prime}) \geq 0  $ for all $ d^{\prime} $ that
	\begin{align}
		& \liminf_{\substack{ \tau \downarrow 0 \\ d^{\prime} \rightarrow d  } } \frac{ \Theta(z+\tau d^{\prime} ) - \Theta (z) }{ \tau^{2} /2 } \notag \\
		\geq \, & \liminf_{\substack{ \tau \downarrow 0 \\ d^{\prime} \rightarrow d  } } \frac{ \Theta(z+\tau d^{\prime} ) - \Theta (z) - \tau \Theta^{\prime} (z;d^{\prime}) }{ \tau^{2} /2 } \notag \\
		\geq \, & \liminf_{\substack{ \tau \downarrow 0 \\ d^{\prime} \rightarrow d  } } \frac{ F(z+\tau d^{\prime} ) - F (z) - \tau F^{\prime} (z;d^{\prime}) }{ \tau^{2} /2 } \label{eq:4-11} \\
		&+ \sum_{\ell=1}^{L} \beta_{\ell} \sum_{j \in [N_{\ell}]} \liminf_{\substack{ \tau \downarrow 0 \\ d^{\prime} \rightarrow d  } } \frac{ f_{\ell,j}(z+\tau d^{\prime} ) - f_{\ell,j} (z) - \tau f_{\ell,j}^{\prime} (z;d^{\prime}) }{ \tau^{2} /2 } ,\notag
	\end{align}
	where $ f_{\ell,j}(z):= | [u_{\ell}]_{j} - [\psi_{\ell-1}]_{j}( \theta, \mathbf{u}_{\ell-1} )  | $ for all $ \ell \in [L], j \in [N_{\ell}] $. By the twice semidifferentiability of $g$ and $ \lambda \|\cdot\|^{2} $, it follows from \eqref{eq:4-9} that
	\begin{align}
		\label{eq:4-12}
		\liminf_{\substack{ \tau \downarrow 0 \\ d^{\prime} \rightarrow d  } } \frac{ F(z+\tau d^{\prime} ) - F (z) - \tau F^{\prime} (z;d^{\prime}) }{ \tau^{2} /2 }
		= F^{(2)}(z;d).
	\end{align}
	And for all $ \ell \in [L], j\in [N_{\ell}] $, it follows from twice semidifferentiability of $ [\psi_{\ell-1}]_{j} $ and \cite[(4.15)]{CP-book} that
	\begin{align*}
		& \liminf_{\substack{ \tau \downarrow 0 \\ d^{\prime} \rightarrow d  } } \frac{ f_{\ell,j}(z+\tau d^{\prime} ) - f_{\ell,j} (z) - \tau f_{\ell,j}^{\prime} (z;d^{\prime}) }{ \tau^{2} /2 }  \\
		\geq \,& \left\{
		\begin{matrix}
			- [ \psi_{\ell-1}^{(2)} ( \theta, \mathbf{u}_{\ell-1} ; d_{\theta} , d_{ \mathbf{u}_{\ell-1} } ) ]_{j}, & \text{if } j \in I_{+}^{\ell}(z) \cup I_{0,+}^{\ell}(z;d), \\
			[ \psi_{\ell-1}^{(2)}( \theta, \mathbf{u}_{\ell-1} ; d_{\theta} , d_{ \mathbf{u}_{\ell-1} } ) ]_{j}, & \text{if } j \in I_{-}^{\ell}(z) \cup I_{0,-}^{\ell}(z;d) ,\\
			-| [ \psi_{\ell-1}^{(2)}( \theta, \mathbf{u}_{\ell-1} ; d_{\theta} , d_{ \mathbf{u}_{\ell-1} } ) ]_{j} |, & \text{if } j \in  I_{0,0}^{\ell}(z;d) ,
		\end{matrix}
		\right.
	\end{align*}
	where $ I_{+}^{\ell}(z)$, $ I_{-}^{\ell}(z)$, $I_{0}^{\ell}(z) $ are defined as in Proposition \ref{dd_F1}, $ I_{0,+}^{\ell}(z;d)$, $I_{0,-}^{\ell}(z;d)$, $I_{0,0}^{\ell}(z;d)  $ are defined as in Proposition \ref{2dd_F1}.
	In fact, $ I_{+}^{\ell}(z) =  I_{-}^{\ell}(z) = \emptyset $ for all $ \ell $ since $ z \in \mathcal{F}_{0} $ according to Lemma \ref{lem:D1-F0}, and furthermore it follows from \eqref{eq:4-8-0} that $I_{0,+}^{\ell}(z;d) = I_{0,-}^{\ell}(z;d) = \emptyset $ for all $ \ell $. Thus, the inequality can be simplified as
	\begin{align}
		\label{eq:4-13}
		\liminf_{\substack{ \tau \downarrow 0 \\ d^{\prime} \rightarrow d  } } \frac{ f_{\ell,j}(z+\tau d^{\prime} ) - f_{\ell,j} (z) - \tau f_{\ell,j}^{\prime} (z;d^{\prime}) }{ \tau^{2} /2 }
		\geq   -| [ \psi_{\ell-1}^{(2)}( \theta, \mathbf{u}_{\ell-1} ; d_{\theta} , d_{ \mathbf{u}_{\ell-1} } ) ]_{j} |.
	\end{align}
	Plugging \eqref{eq:4-12} and \eqref{eq:4-13} into \eqref{eq:4-11}, we have
	\begin{align*}
		\liminf_{\substack{ \tau \downarrow 0 \\ d^{\prime} \rightarrow d  } } \frac{ \Theta(z+\tau d^{\prime} ) - \Theta (z) }{ \tau^{2} /2 }
		\geq F^{(2)}(z;d) - \sum_{\ell=1}^{L} \beta_{\ell} \| \psi_{\ell-1}^{(2)}( \theta, \mathbf{u}_{\ell-1} ; d_{\theta} , d_{ \mathbf{u}_{\ell-1} } ) \|_{1}
		>0 ,
	\end{align*}
	where the last inequality comes from \eqref{eq:suf-con}. Thus, \eqref{eq:4-10} holds and $z$ is a strong local minimizer of \eqref{l1pen}.
	\\
	\indent Next we prove $\theta$, the component of $z$, is a strong local minimizer of \eqref{eq:1.1} by contradiction.
	If $\theta$ is not a strong local minimizer of \eqref{eq:1.1}, 
	then there exists a sequence $ \{ \theta^{k}, k\geq 1 \} $ converging to $ \theta $ such that
	\begin{align*}
		\Psi( \theta^{k} ) + \lambda \| \theta^{k} \|^{2} < \Psi( \theta  ) + \lambda \| \theta  \|^{2}  + \frac{1}{k} \| \theta^{k} - \theta  \|^{2}.
	\end{align*}
	Based on $ \{ \theta^{k}, k\geq 1 \} $, we can construct $ \{ z^{k}, k\geq 1 \} \subseteq \mathcal{F}_{0} $ by setting $ z^{k}=( (\theta^{k})^{\top} , (u_{1}^{k})^{\top} , \dots  , (u_{L}^{k})^{\top} )^{\top} $ with $ u_{\ell}^{k}:= \psi_{\ell-1} (\theta^{k}, u_{1}^{k},\dots,u_{\ell-1}^{k} ) $ for all $ \ell \in [L] $. Together with $ z \in \mathcal{F}_{0} $ from Lemma \ref{lem:D1-F0}, it implies that
	\begin{align}
		\Theta(z^{k})
		= \Psi( \theta^{k} ) + \lambda \| \theta^{k} \|^{2}
		< \Psi( \theta  ) + \lambda \| \theta  \|^{2}  + \frac{1}{k} \| \theta^{k} - \theta  \|^{2}
		&= \Theta(z) + \frac{1}{k} \| \theta^{k} - \theta  \|^{2} \notag \\
		&\leq \Theta(z) + \frac{1}{k} \| z^{k} - z  \|^{2}. \label{eq:4-14}
	\end{align}
	Meanwhile, it follows from the continuity of $ \{ \psi_{\ell-1}, \ell \in [L] \} $ and $ \theta^{k} \rightarrow \theta $ that $ z^{k} \rightarrow z $. Together with the strict inequality in \eqref{eq:4-14}, it contradicts the fact that $z$ is a strong local minimizer of \eqref{l1pen}. Hence, $\theta$ is a strong local minimizer of \eqref{eq:1.1}.
\end{proof}
\begin{remark}
	\label{rem:suf-con}
	According to \eqref{eq:4-8-0}, the sufficient condition \eqref{eq:suf-con} is equivalent to $ \Theta^{\prime}(z;d) \geq 0 $ for all $d$ and $ F^{(2)} (z;d) - \sum_{\ell=1}^{L} \beta_{\ell}  \| \psi_{\ell-1}^{(2)} ( \theta, \mathbf{u}_{\ell-1} ; d_{\theta} , d_{\mathbf{u}_{\ell-1}} ) \|_{1} > 0 $ for all $ d \neq \mathbf{0} $ with $ d \in \mathcal{T}_{\mathcal{F}_{0}}(z)$ and $ F^{\prime}(z;d) = 0 $. Hence, it can be observed that for the case $ L=1 $, \eqref{eq:suf-con} is milder than \cite[(9.42)]{CP-book} for $ \min_{\theta} h(G(\theta)) $ with $ G(\theta):= ( \psi_{0}(\theta)^{\top} , \lambda \|\theta \|^{2} )^{\top}$ and $ h(y):= g( [y]_{1:N_{1}} ) + [y]_{(N_{1} + 1) } $, since under $ \rho:=\beta_{1} $ and $J_{+}:=\{ j \in [N_{1}] \mid [\psi_{0}^{(2)} (\theta ; d_{\theta}) ]_{j} \leq 0 \}, J_{-}:=\{N_{1} +1 \} \cup ( [N_{1}] \backslash J_{+} ) $,
	\begin{align*}
		& h^{(2)}( G(\theta);G^{\prime}(\theta; d_{\theta}) ) + \rho \,   \left[ \sum_{j \in J_{+}} G_{j}^{(2)}(\theta;d_{\theta}) -\sum_{j \in J_{-}} G_{j}^{(2)}(\theta;d_{\theta})  \right] \\
		= \, & g^{(2)} ( \psi_{0}(\theta);\psi_{0}^{\prime}(\theta; d_{\theta}) ) - 2 \lambda \beta_{1} \|d_{\theta}  \|^{2} - \beta_{1}  \| \psi_{0}^{(2)} ( \theta ; d_{\theta}  ) \|_{1} \\
		< \, & F^{(2)} (z;d) -   \beta_{1}  \| \psi_{0}^{(2)} ( \theta ; d_{\theta}  ) \|_{1},
	\end{align*}
	where the inequality uses $ F^{(2)} (z;d)= g^{(2)} ( \psi_{0}(\theta);\psi_{0}^{\prime}(\theta; d_{\theta}) ) + 2 \lambda  \|d_{\theta}  \|^{2} $ and $ d_{\theta} \neq \mathbf{0} $ for all $ d \in \mathcal{T}_{\mathcal{F}_{0}}(z) $ with $ d \neq \mathbf{0} $. 
\end{remark}
Theorem \ref{th:suf-con} enables us to determine whether a d-stationary point of \eqref{l1pen} is a strong local minimizer for \eqref{eq:1.1}.

\section{Application: RNNs} \label{sec4}
The recurrent neural network (RNN) is a kind of feedforward neural networks for sequential processing. Different RNN variants, such as Elman networks \cite{ELMAN1990}, Jordan networks \cite{J1990}, and LSTM \cite{LSTM1997}, have been widely applied on language modelling like ChatGPT and protein secondary structure prediction \cite{Graves2012SupervisedSL}. 
Due to the universal approximation property and the fundamental significance for the other RNN variants \cite{HAMMER2000}, we focus on the training of the Elman RNN with a single unidirectional hidden layer in this section.
Without loss of generality, we consider the case where the number of sequences is $ N=1 $ and the number of time steps in the sequence is $T=3$.
Given a sequence of inputs $ \{ x_{t}\in \R^{N_{0}} , t \in [3] \}$ and an associated sequence of labels $ \{ y_{t}\in \R^{N_{2}} , t \in [3] \}$, the model can be formulated as the following constrained optimization problem
\begin{equation}
	\label{form2-eq}
	\tag{P0-RNN}
	\begin{aligned}
		\min_{\substack{A,V,W,b,c,\\  s,w,r,v }}&~   \frac{\left\|  r - y \right\|^{2}}{6}     + \lambda  (\left\| A \right\|_{F} ^{2}+\left\| V \right\|_{F} ^{2}+\left\| W \right\|_{F} ^{2}+ \left\| b \right\|^{2} + \left\| c \right\|^{2}   )    ,   \\
		\text{subject to }
		&~w_{t} =  W s_{t-1}  +A x_{t}  + b, ~  s_{t} =\sigma   \left( w_{t}  \right) , \\
		&~ v_{t} =  V s_{t}  + c,~ r_{t} =\sigma  \left( v_{t}  \right) ,~   t \in [3],
	\end{aligned}
\end{equation}
where $ s_{0}  =\mathbf{0} \in \R^{N_{1}} $  and other notations are defined as follows.
\begin{enumerate}
	\item Vector $ y  $ refers to $ y =(  y_{1} ^{\top}, y_{2} ^{\top},  y_{3} ^{\top} )^{\top} \in \R^{3 N_{2} } $.
	
	\item Vectors $ s_{t}  \in \R^{N_{1}} $ and $ r_{t}  \in \R^{N_{2}} $ refer to the hidden value and output at time $ t $, respectively. For brevity, we denote $ s =(  s_{1} ^{\top},s_{2} ^{\top}, s_{3} ^{\top} )^{\top} \in \R^{3 N_{1} },~
	r =(  r_{1} ^{\top}, r_{2} ^{\top}, r_{3} ^{\top} )^{\top} \in \R^{3 N_{2}  } $.
	
	\item Vectors $ w_{t}  \in \R^{N_{1}} $ and $ v_{t}  \in \R^{N_{2}} $ refer to the auxiliary hidden value and auxiliary output at time $ t $, respectively. We denote $ w =(  w_{1} ^{\top},w_{2} ^{\top}, w_{3} ^{\top} )^{\top} \in \R^{3 N_{1}  }  , ~
	v =(  v_{1} ^{\top},v_{2} ^{\top}, v_{3} ^{\top} )^{\top} \in \R^{3 N_{2}  } $.
	
	\item Matrices $ W\in \R^{N_{1}\times N_{1}}, A\in \R^{N_{1}\times N_{0}},V\in \R^{N_{2}\times N_{1}}$ and vectors $ b\in \R^{N_{1}},c\in \R^{N_{2}} $ are network parameters  independent of $t$. And we aggregate those parameters as
	\begin{align}\label{def:theta}
		\theta:= \left( \operatorname{vec}(A)^{\top} ,\operatorname{vec}(V)^{\top} ,\operatorname{vec}(W)^{\top} ,b^{\top} ,c^{\top}\right) ^{\top} \in \R^{n} ,
	\end{align}
	where $\operatorname{vec} (A):= ( a_{1}^{\top} , \dots, a_{q}^{\top}
	)^{\top} $ for any matrix $A=( a_{1}, \dots, a_{q} ) \in \R^{p \times q} $ with $ \{ a_{j} \in \R^{p}, j \in [q] \} $, and $ n:= N_{0}N_{1}+N_{1}N_{2}+N_{1}^{2}+N_{1}+ N_{2} $ in this case.
	
	\item Function $\sigma(u) :=\max \{ u, \alpha u \}$ for all $  u \in \R $ is (leaky) ReLU activator with $ \alpha \in [0,1) $. For brevity, we will not distinguish whether $ \sigma(\cdot) $ applies on a scalar or on a vector componentwisely when there is no ambiguity.
\end{enumerate} 

To reconcile the notations in \eqref{form2-eq} with those in \eqref{eq}, we could first define $ L:= 8 $ and
\begin{align*}
	u_{2t-1}:= w_{t}, ~ u_{2t}:= s_{t}, ~ t \in [3], ~~
	u_{7}:= v, ~ u_{8}:= r;
\end{align*}
then define $ \mathbf{u}_{0} $ to be an empty placeholder, $ \mathbf{u}_{\ell}:= ( u_{1}^{\top},\dots,u_{\ell}^{\top} )^{\top}  $ for all $ \ell \in [8] $ as in \eqref{def:1.1}. Thereby, we have
\begin{align*}
	u:=\mathbf{u}_{L}
	=(w_{1}^{\top}, s_{1}^{\top},w_{2}^{\top}, s_{2}^{\top},w_{3}^{\top}, s_{3}^{\top}, v^{\top}, r^{\top} )^{\top}
	\in \R^{ 6( N_{1} + N_{2}  ) } ,
\end{align*}
which aggregates all the auxiliary variables $s,w,r,v $ in \eqref{form2-eq}.
Together with \eqref{def:theta},  we have
\begin{align}
	z:=(\theta^{\top} ,  u^{\top} )^{\top}\in \R^{\bar{N}} ,\text{ where }
	\bar{N}:= n +6 (N_{2}+N_{1}),  \label{def:zrnn}
\end{align}
so that the objective function of \eqref{form2-eq} can be denoted as
\begin{align*}
	F\left( z\right) := g(u) + \lambda  \left\| \theta \right\| ^{2}, \text{ where }
	g(u):= \left\|  r - y \right\|^{2} / 6.
\end{align*} 
And for the constraints of \eqref{form2-eq}, using Kronecker product $ \otimes $, we denote
\begin{align*}
	\psi_{2t-2}(\theta, \mathbf{u}_{2t-2} ) 
	:= ( x_{t}^{\top} \otimes I_{N_{1}} \quad \mathbf{0} \quad  s_{t-1}^{\top} \otimes I_{N_{1}} \quad  I_{N_{1} }  \quad  \mathbf{0} ) \, \theta ,  ~
	\psi_{2t-1}(\theta, \mathbf{u}_{2t-1} ) 
	:= \sigma(w_{t} )
\end{align*}
for all $ t \in [3] $, and
\begin{align*}
	\psi_{6}(\theta,\mathbf{u}_{6}) 
	:=
	\begin{pmatrix}
		\mathbf{0} & s_{1}^{\top} \otimes I_{N_{2}}  & \mathbf{0} & \mathbf{0} & I_{N_{2}} \\
		\mathbf{0} & s_{2}^{\top} \otimes I_{N_{2}}  & \mathbf{0} & \mathbf{0} & I_{N_{2}} \\
		\mathbf{0} & s_{3}^{\top} \otimes I_{N_{2}}  & \mathbf{0} & \mathbf{0} & I_{N_{2}}
	\end{pmatrix} \theta , ~
	\psi_{7}( \theta, \mathbf{u}_{7}) 
	:= \sigma(v).
\end{align*}
Then it can be checked that
\begin{align*}
	u_{\ell}= \psi_{\ell-1}(\theta, \mathbf{u}_{\ell-1} ), \, \ell \in [8] \Leftrightarrow
	\left \{
	\begin{aligned}
		& w_{t} =  W s_{t-1}  +A x_{t}  + b, ~  s_{t} =\sigma   \left( w_{t}  \right) , \\
		&  v_{t} =  V s_{t}  + c,~ r_{t} =\sigma  \left( v_{t}  \right) ,~   t \in [3],
	\end{aligned} \right.
\end{align*}
and the above functions $g $ and $ \{ \psi_{\ell-1}, \ell \in [8] \} $ are continuous. Hence, \eqref{form2-eq} is an application of \eqref{eq}.
Naturally, \eqref{form2-eq} has a reformulation corresponding to \eqref{eq:1.1}. As noted at the beginning of Section \ref{sec3.2}, they are equivalent when neglecting dimension lifting.
\begin{remark}\label{remark-[8]}
	\eqref{form2-eq} provides an example illustrating the differences between \eqref{eq:1.1} and (2.1)-(2.2) of \cite{DNN-CHP}.
	Firstly, unifying $ A,V,W,b,c $ as $ \theta $ makes it convenient to sharing parameters $ A,W,b $ in $ \psi_{0},\psi_{2} $ and $ \psi_{4} $.
	Secondly, $ \psi_{6} $ not only depends on $ \theta $ and $ u_{6} $ (i.e., $ s_{3} $), but also depends on $ u_{2}, u_{4} $ (i.e., $ s_{1}, s_{2} $), which transmits the information across multiple layers.
	In contrast, DNNs in \cite{DNN-CHP} demand distinct parameters in different layers, which lacks a mechanism to maintain parameter consistency among the layers sharing parameters during the training process. Figure \ref{fig1} shows the architectures of RNN in \eqref{form2-eq} and DNN in (1.1) of \cite{DNN-CHP}.  
\end{remark}

\begin{figure}[htbp]
	\centering 
	\subfigure[RNN]{
		\label{RNN}
		\includegraphics[width=0.4 \textwidth]{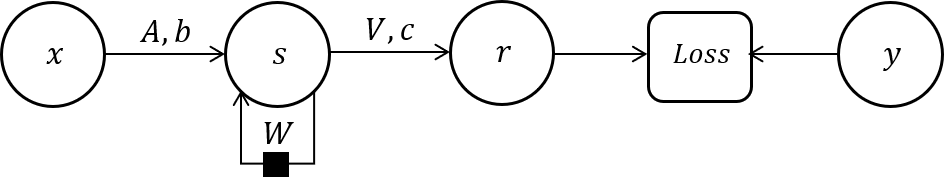}}

	\subfigure[DNN]{
		\label{DNN}
		\includegraphics[width=0.7 \textwidth]{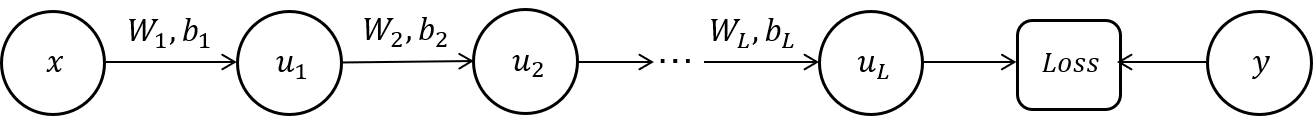}}
	\caption{Architectures of RNN in \eqref{form2-eq} and DNN in \cite{DNN-CHP}}
	\label{fig1}
\end{figure}

Next, we will explore the benefits of results in Section \ref{sec3} for RNN training based on \eqref{form2-eq}.
For simplicity, we merge penalty parameters $ \{ \beta_{\ell}, \ell \in [L] \} $ into $ (\beta_{1} ,\beta_{2} )>0 $ in the $\ell_{1}$-penalized form of \eqref{form2-eq}:
\begin{equation}
	\label{form2-l1pen}
	\tag{P1-RNN}
	\begin{aligned}
		\min_{z \in \R^{\bar{N}}}
		\Theta(z) ,
	\end{aligned}
\end{equation}
where
\begin{align*}
	\Theta(z):= \, &F(z) + \beta_{1} \sum_{\ell=1}^{6}  \left\| u_{\ell}-\psi_{\ell-1}(\theta, \mathbf{u}_{\ell-1} )  \right\| _{1}  + \beta_{2} \sum_{\ell=7}^{8}  \left\| u_{\ell}-\psi_{\ell-1}(\theta, \mathbf{u}_{\ell-1} )  \right\| _{1}  \\
	= \, & F(z) + \beta_{1} \sum_{t=1}^{3}  \left( \| w_{t}  - W s_{t-1}  -A x_{t}  - b \|_{1} + \| s_{t}  - \sigma( w_{t}  ) \|_{1} \right) \\
	& + \beta_{2} \sum_{t=1}^{3}   \left(  \| v_{t} - V s_{t }  - c \|_{1} + \| r_{t}  - \sigma( v_{t} ) \|_{1} \right).
\end{align*}

For \eqref{form2-eq} and \eqref{form2-l1pen}, denote
\begin{align*}
	&\mathcal{F}_{0}^{RNN}:= \left\{ z \Bigg|
	\begin{matrix}
		w_{t} =  W s_{t-1}  +A x_{t}  + b,~
		s_{t} =\sigma   \left( w_{t}  \right) , \\
		v_{t} =  V s_{t}  + c, ~
		r_{t} =\sigma  \left( v_{t}  \right) ,  ~ t \in [3]
	\end{matrix} \right\} , \\
	&\mathcal{S}_{0}^{RNN} := \argmin_{z \in \mathcal{F}_{0}^{RNN} } F(z), ~ \mathcal{S}_{1}^{RNN} := \argmin_{z \in \R^{\bar{N}} } \Theta (z) ,
\end{align*}
and break the direction $d \in \R^{\bar{N}} $ according to the blocks of variable $z$ defined in \eqref{def:zrnn}:
\begin{equation}
	\label{def:drnn}
	\begin{aligned}
		&d=(d_{\theta}^{\top},  d_{u}^{\top}  )^{\top}, \text{ with } \\
		&~d_{\theta}=\left(d_{A}^{\top} ,d_{V}^{\top} ,d_{W}^{\top} ,d_{b}^{\top} , d_{c}^{\top}  \right) ^{\top}  , \\
		&~d_{u}=(d_{u_{1}}^{\top}, \dots, d_{u_{8}}^{\top} )^{\top}= ( d_{w_{1}}^{\top}, d_{s_{1}}^{\top},d_{w_{2}}^{\top}, d_{s_{2}}^{\top}, d_{w_{3}}^{\top}, d_{s_{3}}^{\top}, d_{v}^{\top}, d_{r}^{\top} )^{\top} , 
	\end{aligned}
\end{equation}
where all dimensions are consistent with the corresponding variables. 
And for any $d \in \R^{\bar{N}} $, we define $ d_{r}:= ( d_{r_{1}}^{\top} , d_{r_{2}}^{\top} , d_{r_{3}}^{\top}  )^{\top}  $, $ d_{v}:= ( d_{v_{1}}^{\top} , d_{v_{2}}^{\top} , d_{v_{3}}^{\top}  )^{\top}   $, $ d_{s}:= ( d_{s_{1}}^{\top},  d_{s_{2}}^{\top},d_{s_{3}}^{\top}  )^{\top} $, $ d_{w}:= ( d_{w_{1}}^{\top},  d_{w_{2}}^{\top},d_{w_{3}}^{\top}  )^{\top} $, and $ d_{s_{0} }:= \mathbf{0} \in \R^{N_{1}} $. 
For any vector $a \in \R^{pq} $, denote $$ \operatorname{mat}_{p,q}(a):=
\begin{pmatrix}
	a_{1}& a_{p+1}& \cdots& a_{p(q-1)+1}\\
	\vdots& \vdots &\vdots& \vdots \\
	a_{p} & a_{2p} & \cdots& a_{pq}
\end{pmatrix}. $$ 
Applying the results in Sections \ref{sec2} and \ref{sec3} on \eqref{form2-eq} and \eqref{form2-l1pen}, we obtain the following corollary.
\begin{corollary}\label{cor:rnn}
	(i) The optimal solution sets $ \mathcal{S}_{0}^{RNN} $ and $ \mathcal{S}_{1}^{RNN} $ are nonempty and compact for all $ ( \beta_{1} , \beta_{2} ) >0 $.
	(ii) Any local minimizer $z$ of \eqref{form2-eq} is a second-order d-stationary point of \eqref{form2-eq}, that is
	\begin{align*}
		& z \in \mathcal{F}_{0}^{RNN}
		\text{ and } [\nabla F(z)]^{\top} d \geq 0, \, \forall d \in \mathcal{T}_{\mathcal{F}_{0}^{RNN}}(z) \\
		& \text{and } d^{\top} \nabla^{2} F(z) d \geq 0 , \, \forall d \in P_{\mathcal{F}_{0}^{RNN}}  (z) \text{ with } [\nabla F(z)]^{\top} d =0,
	\end{align*}
	and for all $ ( \beta_{1} , \beta_{2} ) >0 $, any local minimizer $z$ of \eqref{form2-l1pen} is a second-order d-stationary point of \eqref{form2-l1pen}, that is
	\begin{align*}
		\Theta^{\prime}(z;d)\geq 0, \, \forall d \in \R^{\bar{N}}
		\text{and } \Theta^{(2)}(z;d) \geq 0 , \, \forall d \text{ with } \Theta^{\prime}(z;d) =  0,
	\end{align*}
	where for any $z \in \mathcal{F}_{0}^{RNN} $,
	\begin{align*}
		&\mathcal{T}_{\mathcal{F}_{0}^{RNN}}(z)= \left\{ d \in \R^{\bar{N}} \Bigg|
		\begin{matrix}
			d_{v_{t}}= D_{V} s_{t}+ V d_{ s_{t} } + d_{c}, ~
			d_{r_{t}} = \sigma^{\prime}(v_{t}; d_{v_{t}}), \\
			d_{w_{t}}= D_{W} s_{t-1}+ W d_{ s_{t-1} }  + D_{A} x_{t} + d_{b}, \\
			d_{s_{t}} = \sigma^{\prime}(w_{t}; d_{w_{t}}),
			~ t \in [3]
		\end{matrix}
		\right\} , \\
		& P_{\mathcal{F}_{0}^{RNN}} (z)= \{ d \in \mathcal{T}_{\mathcal{F}_{0}^{RNN}}(z) \mid D_{V} d_{ s_{t} } =0, ~ D_{W} d_{ s_{t-1} } = 0, ~ t \in [3] \}
	\end{align*}
	with $ D_{A}:= \operatorname{mat}_{N_{1}, N_{0}}(  d_{A} )  ,\,
	D_{V}:= \operatorname{mat}_{N_{2}, N_{1}}(  d_{V} )  ,\,
	D_{W}:= \operatorname{mat}_{N_{1}, N_{1}}(  d_{W} )  $. \\
	(iii) Under the setting of
	\begin{align} \label{eq:br-threshold-rnn}
		\beta_{1} >    \gamma_{1} \gamma_{y}  \sqrt{{2}/{(3 \lambda)}}    ,  ~ ~
		\beta_{2} >  \sqrt{{2 \gamma_{y}}/{3}} ,
	\end{align}
	where
	$
	\gamma_{y}:= \Theta(\mathbf{0}) = \|y\|^{2} / 6,
	\gamma_{1}:= \sum_{i=0}^{2} (   \sqrt{{\gamma_{y}}/{\lambda} }  ) ^{i}
	$, we have that
	\begin{itemize}
		\item[(a)] $ \mathcal{S}_{0}^{RNN} = \mathcal{S}_{1}^{RNN} $;  and
		\item[(b)] for any $  z\in lev_{\leq \gamma_{y}} \Theta $, $z$ is d-stationary point of \eqref{form2-eq} if and only if it is a d-stationary point of \eqref{form2-l1pen}.
	\end{itemize}
\end{corollary}
\begin{proof}
	$(i)$.  The nonemptiness and compactness of $ \mathcal{S}_{0}^{RNN} $ and $ \mathcal{S}_{1}^{RNN} $ come from Lemmas \ref{lem:bd-opt-0} and \ref{lem:bd-opt-1}.
	
	$(ii)$.  Since $ g $ and $ \{ \psi_{\ell-1}, \ell \in [8] \} $ satisfy Assumptions \ref{as1} and \ref{as2}, we attain the necessity of second-order d-stationarity from Lemma \ref{lem:1op}, and the expression of $ \mathcal{T}_{\mathcal{F}_{0}^{RNN}}(z) $ from Theorem \ref{th:T-express-2}. The expression of $ P_{\mathcal{F}_{0}^{RNN}} (z) $ is further derived by its definition, $ P_{\mathcal{F}_{0}^{RNN}}(z) \subseteq \mathcal{T}_{\mathcal{F}_{0}^{RNN}} (z) $ and \eqref{eq:3.3.1}.
	
	$(iii)$.  The thresholds \eqref{eq:br-threshold-rnn} can be obtained by refining the proof of Lemma \ref{lem:D1-F0} and Theorem \ref{th:P0=P1} as follows. Note that
	\begin{align*}
		\|r-y\| \leq \sqrt{ 6 \gamma_{y} }, ~
		\|V\|\leq \sqrt{ \gamma_{y}/\lambda }, ~
		\|W\|\leq \sqrt{ \gamma_{y}/\lambda } , 
	\end{align*}
	for all $z \in lev_{\leq \gamma_{y}} \Theta $. Next, we will estimate the constants $ K_{g}>0 $ and $ \{ K_{\ell}>0, \ell \in [L-1] \} $ satisfying \eqref{eq:Lip}.
	By the definition of $ g(u)  $ in \eqref{form2-eq} and \eqref{form2-l1pen}, it implies that
	\begin{align}
		| g^{\prime}(u;  d_{u}  ) - g^{\prime}( u ;  \bar{d}_{u} )|
		=  | \frac{ (r-y)^{\top} ( d_{r}- \bar{d}_{r} ) }{3} |
		\leq\,& \frac{\| r-y \| }{3} \| d_{r}- \bar{d}_{r} \| \notag\\
		\leq \,& \sqrt{ \frac{2 \gamma_{y} }{3} } \| d_{r}- \bar{d}_{r} \| \label{eq:5-1}
	\end{align}
	for any $z \in lev_{\leq \gamma_{y}} \Theta $ and any $ d,\bar{d} \in \R^{\bar{N}} $.
	Then for $ \{ \psi_{\ell}, \ell \in [7] \} $, we divide them into four groups. For $ \{\psi_{2t-2}, t = 2,3 \} $, it follows from $ \|W\|\leq \sqrt{\gamma_{y}/\lambda} $ for all $ z \in lev_{\leq \gamma_{y} } \Theta $ that
	\begin{equation}
		\label{eq:5-2}
		\begin{aligned}
			& \| \psi_{2t-2}^{\prime}(\theta,\mathbf{u}_{2t-2}; d_{\theta},  d_{ \mathbf{u}_{2t-2} }   )  - \psi_{2t-2}^{\prime}(\theta,\mathbf{u}_{2t-2}; \bar{d}_{\theta}, \bar{d}_{ \mathbf{u}_{2t-2} }  )   \|  \\
			= \, &  \left\| W (d_{ s_{t-1}  } - \bar{d}_{ s_{t-1}  } )\right\| 
			\leq   \sqrt{\gamma_{y}/\lambda} \| d_{s_{t-1}}- \bar{d}_{s_{t-1}}  \|
		\end{aligned}
	\end{equation}
	for all $ d,\bar{d} \in \R^{\bar{N}} $ with $ d_{\theta} = \bar{d}_{\theta} $.
	For $ \{ \psi_{2t-1}, t \in [3] \} $, it follows from the definition of $ \sigma $ that
	\begin{equation}
		\label{eq:5-3}
		\begin{aligned}
			& \| \psi_{2t-1}^{\prime}(\theta,\mathbf{u}_{2t-1}; d_{\theta},  d_{ \mathbf{u}_{2t-1} }  )   - \psi_{2t-1}^{\prime}(\theta,\mathbf{u}_{2t-1}; \bar{d}_{\theta},  \bar{d}_{ \mathbf{u}_{2t-1} }  )   \| \\
			= \, &  \| \sigma^{\prime}( w_{t} ; d_{ w_{t} } ) - \sigma^{\prime}( w_{t} ; \bar{d}_{ w_{t} } ) \|
			\leq \| d_{ w_{t} } - \bar{d}_{ w_{t} } \|
		\end{aligned}
	\end{equation}
	for all $ d,\bar{d} \in \R^{\bar{N}} $.
	For $ \psi_{6} $, it follows from $ \|V\|\leq \sqrt{\gamma_{y}/\lambda} $ that for all $ z \in lev_{\leq \gamma_{y} } \Theta $ and for all $ d,\bar{d} \in \R^{\bar{N}} $ with $ d_{\theta} = \bar{d}_{\theta} $,
	\begin{equation}
		\label{eq:5-4}
		\begin{aligned}
			& \| \psi_{6}^{\prime}(\theta, \mathbf{u}_{6}; d_{\theta},   d_{ \mathbf{u}_{6} }    )  -  \psi_{6}^{\prime}(\theta, \mathbf{u}_{6} ; \bar{d}_{\theta},  \bar{d}_{ \mathbf{u}_{6} }   )   \|   \\
			= \, & \left\|
			\begin{pmatrix}
				V (d_{s_{1} } - \bar{d}_{s_{1} }) \\
				V (d_{s_{2} } - \bar{d}_{s_{2} }) \\
				V (d_{s_{3} } - \bar{d}_{s_{3} })
			\end{pmatrix}
			\right\|
			\leq    \sqrt{\gamma_{y}/\lambda} \sum_{ t\in [3]}\| d_{s_{t} } - \bar{d}_{s_{t} } \|.
		\end{aligned}
	\end{equation}
	For $ \psi_{7} $, it follows from the definition of $\sigma$ that
	\begin{align}
		& \| \psi_{7}^{\prime}(\theta,\mathbf{u}_{7}  ; d_{\theta},   d_{ \mathbf{u}_{7} } )   - \psi_{7}^{\prime}(\theta,  \mathbf{u}_{7} ; \bar{d}_{\theta},   \bar{d}_{ \mathbf{u}_{7} }  )   \|
		=  \| \sigma^{\prime}( v ; d_{ v } ) - \sigma^{\prime}( v ; \bar{d}_{ v } ) \|
		\leq \| d_{ v } - \bar{d}_{ v } \| \label{eq:5-5}
	\end{align}
	for all $ d,\bar{d} \in \R^{\bar{N}} $.
	Then we can yield (a) and (b) under the thresholds \eqref{eq:br-threshold-rnn} by replacing \eqref{eq:Lip} used in Lemma \ref{lem:D1-F0} and Theorem \ref{th:P0=P1} with \eqref{eq:5-1}-\eqref{eq:5-5} as follows.
	\begin{itemize}
		\item For Lemma \ref{lem:D1-F0}, prove $ u_{\ell}=\psi_{\ell-1}(\theta,\mathbf{u}_{\ell-1}) $ in the order of $ \ell=8,\dots,1 $ by contradiction, but without the use of induction \eqref{eq:4-4}, under same definitions of $\bar{z}$. During the process, plug \eqref{eq:5-1} with $ \bar{d}=\mathbf{0} $ into \eqref{eq:4-0.5} and \eqref{eq:4-1}; in \eqref{eq:4-3} and \eqref{eq:4-4+}, use
		\begin{itemize}
			\item \eqref{eq:5-2} with $ d_{\theta}=\mathbf{0}, \bar{d}=\mathbf{0} $,  $t=2,3,$
			\item \eqref{eq:5-3} with $ d_{\theta}=\mathbf{0}, \bar{d}=\mathbf{0} $,  $t=1,2,3$,
			\item \eqref{eq:5-4} with $ d_{\theta}=\mathbf{0}, \bar{d}=\mathbf{0} $,
			\item \eqref{eq:5-5} with $ d_{\theta}=\mathbf{0}, \bar{d}=\mathbf{0} $.
		\end{itemize}
		\item For Theorem \ref{th:P0=P1}, keep the analysis before \eqref{eq:4-6} unchanged except for plugging \eqref{eq:5-1} into \eqref{eq:4-4.5}. Then repeat \eqref{eq:4-6.5} for $ \ell=2,\dots,8 $ instead of using the induction \eqref{eq:4-6}. During the process, we use \eqref{eq:5-2}, \eqref{eq:5-3}, \eqref{eq:5-4} and \eqref{eq:5-5} when the subscript of $\psi$ belongs to $ \{2,4\} $, $ \{1,3,5\} $, $ \{ 6 \} $ and $ \{ 7 \} $  in \eqref{eq:4-6.5} respectively. The calculations after \eqref{eq:4-6.5} are also kept without changes. 
	\end{itemize} 
\end{proof}

The results in Corollary \ref{cor:rnn} can be easily extended to the case where $ N> 1$, $T>3$ and $ s_{t} ,w_{t} ,r_{t} ,v_{t}  $ aggregate corresponding components for all samples at $t$th time step with thresholds
\begin{align}\label{eq:br-threshold-rnn-2}
	\beta_{1} >    \gamma_{1} \gamma_{y}  \sqrt{ 2/(\lambda N T) }  ,\qquad \beta_{2} >  \sqrt{ 2 \gamma_{y} / (NT) } ,
\end{align}
under $ \gamma_{y}:=\Theta(\mathbf{0})={\|y\|^{2}}/{(2NT)} $, $\gamma_{1}:=  \sum_{i=0}^{T-1} (\sqrt{\gamma_{y}/\lambda})^{i}   $.
And the exact penalty in d-stationarity can be generalized to more scenarios including but not limited to more complicated variants in RNNs (such as LSTM and GRU) with any locally Lipschitz continuous and directionally differentiable activator (such as $\tanh$ and ELU). 

Besides, it follows from the convexity of $F$ in \eqref{form2-eq} and \eqref{form2-l1pen} that for all $ z\in \mathcal{F}_{0}^{RNN} $,
$
d^{\top} \nabla^{2} F(z) d \geq 0
$ for all $  d \in P_{\mathcal{F}_{0}^{RNN}}(z) \cap \{ d \mid [\nabla F(z)]^{\top} d =0 \}.
$
Hence, every d-stationary point of \eqref{form2-eq} is a second-order d-stationary point for \eqref{form2-eq}.
Similarly, according to \eqref{eq:4-8-1}, every d-stationary point of \eqref{form2-l1pen} in $ lev_{\leq \gamma_{y}} \Theta $ is a second-order d-stationary point for \eqref{form2-l1pen} under \eqref{eq:br-threshold-rnn-2}.
In fact, it follows from Remark \ref{remark:SD0=SD1} that $ \mathcal{SD}_{0} \cap lev_{\leq \gamma_{y} } \Theta = \mathcal{SD}_{1} \cap lev_{\leq \gamma_{y} } \Theta =\mathcal{D}_{0} \cap lev_{\leq \gamma_{y} } \Theta = \mathcal{D}_{1} \cap lev_{\leq \gamma_{y} } \Theta $ in this case.
As a consequence, one can obtain a second-order d-stationary point of \eqref{form2-eq} and \eqref{form2-l1pen} by applying the algorithms in \cite{DNN-CHP,nMM-CPS} on \eqref{form2-l1pen} with \eqref{eq:br-threshold-rnn-2}.

\section{Conclusions}
\label{sec5}
The paper investigates a class of nonconvex nonsmooth multicomposite optimization problems \eqref{eq:1.1} with an objective function comprised of a regularization term and a multi-layer composite function with twice directionally differentiable and locally Lipschitz continuous components. The d-stationarity of \eqref{eq:1.1} is hard to attain directly, and its second-order d-stationarity is vague without additional assumptions on the objective function. 
Based on the closed-form expression of the tangent cone  $ \mathcal{T}_{\mathcal{F}_{0}}(\cdot) $, we prove the equivalence between \eqref{eq:1.1}, the constrained form \eqref{eq} and the $\ell_{1}$-penalty formulation \eqref{l1pen} in terms of global optimality and d-stationarity.
The equivalence offers an indirect way to compute the d-stationary points of \eqref{eq:1.1}.
Furthermore, it provides second-order necessary and sufficient conditions for \eqref{eq:1.1} through \eqref{eq} and \eqref{l1pen}.
The theoretical results are also applied to the training process    of recurrent neural networks.

\section*{Acknowledgments}
We would like to thank two referees for their constructive and helpful comments.

\bibliographystyle{abbrv}
\bibliography{references}

\end{document}